\documentclass[a4paper,oneside]{article}
\usepackage[utf8]{inputenc}
\usepackage[a4paper,margin=3cm]{geometry} 
\usepackage{amsmath}
\usepackage{amsthm}
\usepackage{amscd}
\usepackage{amssymb}
\usepackage{MnSymbol}
\usepackage{latexsym}
\usepackage{eucal}
\usepackage{dsfont}
\usepackage{mathtools}
\usepackage{enumitem}
\usepackage{todonotes}
\usepackage{verbatim}

\usepackage{pdflscape}

\usepackage[colorlinks,pdftex]{hyperref}
\hypersetup{
linkcolor=black,
citecolor=black,
pdftitle={},
pdfauthor={Franziska K\"uhn},
pdfkeywords={},
}

\makeatletter
\renewcommand\section{\@startsection{section}{1}{0mm}{-1.5\baselineskip}{\baselineskip}{\normalsize\bfseries\sffamily}}
\renewcommand\subsection{\@startsection{subsection}{1}{0mm}{-\baselineskip}{\baselineskip}{\normalsize\bfseries\sffamily}}
\makeatother

\makeatletter
\def\@fnsymbol#1{\ensuremath{\ifcase#1\or *\or **\or \dagger\or \ddagger\or
   \mathsection\or \mathparagraph\or \|\or \dagger\dagger
   \or \ddagger\ddagger \else\@ctrerr\fi}}

\newlength{\preskip}
\newlength{\postskip}
\makeatother

\newtheoremstyle{theorem}{\preskip}{\postskip}{\itshape}{}{\bfseries}{}
{.5em}{\textbf{\thmname{#1}\thmnumber{ #2} (\thmnote{ #3})}}
\newtheoremstyle{definition}{\preskip}{\postskip}{\normalfont}{0pt}{\bfseries}{}{.5em}{}
\newtheoremstyle{remark}{\preskip}{\postskip}{\normalfont}{0pt}{\bfseries}{}{.5em}{}

\swapnumbers
\theoremstyle{theorem} \newtheorem{thm}{Theorem}[section]
\theoremstyle{theorem} \newtheorem{lem}[thm]{Lemma}
\theoremstyle{theorem} \newtheorem{prop}[thm]{Proposition}
\theoremstyle{theorem} \newtheorem{kor}[thm]{Corollary}
\theoremstyle{definition} 
\theoremstyle{remark} 
\theoremstyle{remark} 
\theoremstyle{definition} 
\theoremstyle{definition} \newtheorem*{ack}{Acknowledgements}
\theoremstyle{remark} \newtheorem{bem}[thm]{Remark}
\theoremstyle{remark} \newtheorem{bems}[thm]{Remarks}
\theoremstyle{definition}  \newtheorem{bsp}[thm]{Example}
\theoremstyle{definition}  
\theoremstyle{definition} \newtheorem*{thmo}{Theorem}
\DeclareMathOperator \id {id}

\DeclareMathOperator \spt {supp}

\newcommand{\I}{\mathds{1}}

\newcommand\fa{\qquad \text{for all \ }}

\newcommand{\cadlag}{c\`adl\`ag }

\newcommand\mc[1] {\mathcal{#1}}
\newcommand\mbb[1] {\mathds{#1}}

\author{%
    Franziska K\"{u}hn\thanks{Institut f\"ur Mathematische Stochastik, Fachrichtung Mathematik, Technische Universit\"at Dresden, 01062 Dresden, Germany, \texttt{franziska.kuehn1@tu-dresden.de}} 
}

\title{Random time changes of Feller processes}

\date{}

\begin{document}

\maketitle

\abstract{\noindent We show that the SDE $dX_t = \sigma(X_{t-}) \, dL_t$, $X_0 \sim \mu$ driven by a one-dimensional symnmetric $\alpha$-stable L\'evy process $(L_t)_{t \geq 0}$, $\alpha \in (0,2]$, has a unique weak solution for any continuous function $\sigma: \mathbb{R} \to (0,\infty)$ which grows at most linearly. Our approach relies on random time changes of Feller processes. We study under which assumptions the random-time change of a Feller process is a conservative $C_b$-Feller process and prove the existence of a class of Feller processes with decomposable symbols. In particular, we establish new existence results for Feller processes with unbounded coefficients. As a by-product, we obtain a sufficient condition in terms of the symbol of a Feller process $(X_t)_{t \geq 0}$ for the perpetual integral $\int_{(0,\infty)} f(X_{s}) \, ds$ to be infinite almost surely. \par \medskip

 \par \medskip

}

\section{Introduction} \label{intro}

Kallenberg \cite{kallenberg} showed that a solution to the SDE \begin{equation}
	dX_t = \sigma(X_{t-}) \, dL_t, \qquad X_0 \sim \mu, \label{intro-eq1}
\end{equation}
driven by a one-dimensional symmetric $\alpha$-stable L\'evy process $(L_t)_{t \geq 0}$ can be written as a random time change of a symmetric $\alpha$-stable Lévy process. Using this idea, Zanzotto \cite{zan02} obtained a necessary and sufficient condition for the SDE \eqref{intro-eq1} to have a unique weak solution. However, his proof relies on local times and is therefore restricted to $\alpha \in (1,2]$. For $\alpha \in (0,1]$ the existence of a unique weak solution to \eqref{intro-eq1} is less well understood. It is known that a weak unique solution exists if $\sigma$ is H\"older continuous, bounded and bounded away from $0$, see Kulik \cite{kulik} and Mikulevicius \& Pragarauskas \cite{mik14}. Moreover, there are several results under the rather restrictive assumption that $x \mapsto x + \sigma(x) u$ is non-decreasing for all $u \in (-1,1)$, see e.\,g.\ \cite{li12,zhao08} and the references therein. Let us mention that the statement ``\emph{For $\alpha \in (0,1)$ the SDE has a pathwise unique (hence weak unique) solution for any bounded continuous function $\sigma$ which is bounded away from $0$}'', which can be found in Komatsu \cite{kom82} and Bass \cite{bass03}, is \emph{wrong}, see \cite{bass04} for a counterexample. In this paper we will show the following result which is new for $\alpha \in (0,1]$.

\begin{thmo}
Let $(L_t)_{t \geq 0}$ be a one-dimensional symmetric $\alpha$-stable L\'evy process, $\alpha \in (0,2]$. For any continuous function $\sigma: \mbb{R} \to (0,\infty)$  which grows at most linearly there exists a unique weak solution to the SDE \begin{equation*}
	dX_t = \sigma(X_{t-}) \, dL_t, \qquad X_0 \sim \mu, 
\end{equation*}
for any initial distribution $\mu$.  
\end{thmo}

For the particular case that $\sigma$ is bounded and Lipschitz continuous, it is known that the unique solution to the SDE is a Feller process with symbol $p(x,\xi) = |\sigma(x)|^{\alpha} \, |\xi|^{\alpha}$, cf.\ Schilling \& Schnurr \cite{schnurr} or \cite{sde}. The idea of the proof of the above theorem is to to construct a Feller process with symbol $p(x,\xi) = |\sigma(x)|^{\alpha} \, |\xi|^{\alpha}$, $x,\xi \in \mbb{R}$, and then to show that the Feller process is the unique weak solution to the SDE. \par
This leads us to the more general question under which assumptions there exists a Feller process with symbol \begin{equation*}
	p(x,\xi) = \varphi(x) q(x,\xi), \qquad x,\xi \in \mbb{R}^d,
\end{equation*}
where $\varphi: \mbb{R}^d \to (0,\infty)$ is a deterministic function and $q(x,\xi)$ the symbol of a Feller process, cf.\ Section~\ref{def}. As we will see in Section~\ref{time}, such symbols are closely related with random time changes of Feller processes. If $q=q(\xi)$ does not depend on $x$, i.\,e.\ $q$ is the characteristic exponent of a L\'evy process, then $p(x,\xi) = \varphi(x) q(\xi)$ is a particular case of a so-called decomposable symbol. We will establish sufficient conditions on $\varphi$ and the symbol $q$ which ensure the existence of a Feller process with symbol $p(x,\xi) = \varphi(x) q(x,\xi)$, cf.\ Theorem~\ref{time-13}. The process will be constructed as a perturbation of a time-changed Feller process. Theorem~\ref{time-13} allows us, in particular, to obtain new existence results for Feller processes with unbounded coefficients. For instance we will show that for any continuous mapping $\varphi: \mbb{R}^d \to (0,\infty)$ such that $\varphi(x) \leq c(1+|x|^{\alpha})$, $\alpha \in (0,2]$, there exists a conservative Feller process with symbol \begin{equation*}
	p(x,\xi) = \varphi(x) |\xi|^{\alpha}, \qquad x,\xi \in \mbb{R}^d,
\end{equation*}
cf.\ Corollary~\ref{time-17} and Corollary~\ref{time-19}. As a by-product of the proof, we obtain a sufficient condition in terms of the symbol of a Feller process $(X_t)_{t \geq 0}$ for the perpetual integral $\int_0^{\infty} f(X_s) \, ds$ to be infinite almost surely, cf.\ Corollary~\ref{time-6}; the result applies, in particular, to L\'evy processes. Moreover, we will prove the well-posedness of $(A,C_c^{\infty}(\mbb{R}^d))$-martingale problems for pseudo-differential operators $A$ with decomposable symbols $p(x,\xi) = \varphi(x) \psi(\xi)$, cf.\ Proposition~\ref{app-3}. \par \medskip

This paper is organized as follows. Basic definitions and notation are introduced in Section~\ref{def}. In Section~\ref{cons} we establish a sufficient condition for the conservativeness of a class of stochastic processes, including Feller processes and solutions to martingale problems. Section~\ref{time} is on random time changes of Feller processes. In particular, we give a sufficient condition such that the random time change of a Feller process is a conservative $C_b$-Feller process, cf.\ Theorem~\ref{time-5}, and establish an existence result for Feller processes with symbols of the form $p(x,\xi) = \varphi(x) q(x,\xi)$, cf.\ Theorem~\ref{time-13}. In Section~\ref{app} we apply the results to prove the weak uniqueness of solutions to SDEs driven by symmetric $\alpha$-stable processes and to derive an existence result for Feller processes with decomposable symbols. In the appendix we give a necessary and sufficient condition for the continuity of the symbol with respect to the space variable in terms of the characteristics of the symbol, cf.\ Theorem~\ref{appen-1}.

\section{Preliminaries} \label{def}

We consider the Euclidean space $\mbb{R}^d$ endowed with the canonical scalar product $x \cdot y = \sum_{j=1}^d x_j y_j$ and the Borel $\sigma$-algebra $\mc{B}(\mbb{R}^d)$. By $B(x,r)$ we denote the open ball of radius $r$ centered at $x$ and by $\overline{B(x,r)}$ its closure. We use $\mbb{R}^d_{\Delta}$ to denote the one-point compactification of $\mbb{R}^d$ and extend functions $f: \mbb{R}^d \to \mbb{R}$ to $\mbb{R}_{\Delta}^d$ by setting $f(\Delta):=0$. If $\tau: \Omega \to [0,\infty]$ is a stopping time with respect to a filtration $(\mc{F}_t)_{t \geq 0}$ on a measurable space $(\Omega,\mc{A})$, then we set $\mc{F}_{\infty} := \sigma(\mc{F}_t; t \geq 0)$ and denote by \begin{equation*}
	\mc{F}_{\tau} := \{A \in \mc{F}_{\infty}; \forall t \geq 0: A \cap \{\tau \leq t\} \in \mc{F}_t\}
\end{equation*}
the $\sigma$-algebra associated with $\tau$. \par
An $E$-valued Markov process $(\Omega,\mc{A},\mbb{P}^x,x \in E,X_t,t \geq 0)$ with \cadlag (right-continuous with left-hand limits) sample paths is called a \emph{Feller process} if the associated semigroup $(T_t)_{t \geq 0}$ defined by \begin{equation*}
	T_t f(x) := \mbb{E}^x f(X_t), \quad x \in E, f \in \mc{B}_b(E) := \{f: E \to \mbb{R}; \text{$f$ bounded, Borel measurable}\}
\end{equation*}
has the \emph{Feller property} and $(T_t)_{t \geq 0}$ is \emph{strongly continuous at $t=0$}, i.\,e. $T_t f \in C_{\infty}(E)$ for all $C_{\infty}(E)$ and $\|T_tf-f\|_{\infty} \xrightarrow[]{t \to 0} 0$ for any $f \in C_{\infty} (E)$. Here, $C_{\infty}(E)$ denotes the space of continuous functions vanishing at infinity. Following \cite{rs98} we call a Markov process $(X_t)_{t \geq 0}$ with \cadlag sample paths a \emph{$C_b$-Feller process} if $T_t(C_b(E)) \subseteq C_b(E)$ for all $t \geq 0$. We will always consider $E=\mbb{R}^d$ or $E=\mbb{R}^d_{\Delta}$. An $\mbb{R}^d_{\Delta}$-valued Markov process with semigroup $(T_t)_{t \geq 0}$ is \emph{conservative} if $T_t \I_{\mbb{R}^d} = \I_{\mbb{R}^d}$ for all $t \geq 0$. \par
If the smooth functions with compact support $C_c^{\infty}(\mbb{R}^d)$ are contained in the domain of the generator $(L,\mc{D}(L))$ of a Feller process $(X_t)_{t \geq 0}$, then we speak of a \emph{rich} Feller process. A result due to von Waldenfels and Courr\`ege, cf.\ \cite[Theorem 2.21]{ltp}, states that the generator $L$ of an $\mbb{R}^d$-valued rich Feller process is, when restricted to $C_c^{\infty}(\mbb{R}^d)$, a pseudo-differential operator with negative definite symbol: \begin{equation*}
	Lf(x) =  - \int_{\mbb{R}^d} e^{i \, x \cdot \xi} q(x,\xi) \hat{f}(\xi) \, d\xi, \qquad f \in C_c^{\infty}(\mbb{R}^d), \, x \in \mbb{R}^d
\end{equation*}
where $\hat{f}(\xi) := \mc{F}f(\xi):= (2\pi)^{-d} \int_{\mbb{R}^d} e^{-ix \xi} f(x) \, dx$ denotes the Fourier transform of $f$ and \begin{equation}
	q(x,\xi) = q(x,0) - i b(x) \cdot \xi + \frac{1}{2} \xi \cdot Q(x) \xi + \int_{\mbb{R}^d \backslash \{0\}} (1-e^{i y \cdot \xi}+ i y\cdot \xi \I_{(0,1)}(|y|)) \, \nu(x,dy). \label{cndf}
\end{equation}
We call $q$ the \emph{symbol} of the rich Feller process $(X_t)_{t \geq 0}$ and $-q$ the symbol of the pseudo-differential operator; $(b,Q,\nu)$ are the \emph{characteristics} of the symbol $q$. For each fixed $x \in \mbb{R}^d$, $(b(x),Q(x),\nu(x,dy))$ is a L\'evy triplet, i.\,e.\ $b(x) \in \mbb{R}^d$, $Q(x) \in \mbb{R}^{d \times d}$ is a symmetric positive semidefinite matrix and $\nu(x,dy)$ a $\sigma$-finite measure on $(\mbb{R}^d \backslash \{0\},\mc{B}(\mbb{R}^d \backslash \{0\}))$ satisfying $\int_{y \neq 0} \min\{|y|^2,1\} \, \nu(x,dy)<\infty$. Any pseudo-differential operator $A$ with negative definite symbol $q$ is \emph{dissipative} on $C_c^{\infty}(\mbb{R}^d)$, i.\,e.\ \begin{equation*}
	\|\lambda f-Af\|_{\infty} \geq \lambda \|f\|_{\infty} \fa \lambda>0, f \in C_c^{\infty}(\mbb{R}^d).
\end{equation*}
We say that a rich Feller process with symbol $q$ has \emph{bounded coefficients} if \begin{equation*}
	\sup_{x \in \mbb{R}^d} \left( |q(x,0)| + |b(x)| + |Q(x)| + \int_{\mbb{R}^d \backslash \{0\}}|y|^2 \wedge 1 \, \nu(x,dy) \right)<\infty.
\end{equation*}
A subspace $\mc{D} \subseteq \mc{D}(L)$ is a \emph{core} for the generator $(L,\mc{D}(L))$ if $(L,\mc{D}(L))$ is the closure of $(L,\mc{D})$ with respect to the uniform norm. \par
A \emph{L\'evy process} $(L_t)_{t \geq 0}$ is a rich Feller process whose symbol $q$ does not depend on $x$. This is equivalent to saying that $(L_t)_{t \geq 0}$ has stationary and independent increments and \cadlag sample paths. The symbol $q=q(\xi)$ (also called \emph{characteristic exponent}) and the L\'evy process $(L_t)_{t \geq 0}$ are related through the L\'evy--Khintchine formula: \begin{equation*}
 \mbb{E}^xe^{i \xi \cdot (L_t-x)} = e^{-t q(\xi)} \fa t \geq 0, \, x,\xi \in \mbb{R}^d.
\end{equation*}
Our standard reference for L\'evy processes is the monograph \cite{sato} by Sato. \emph{Weak uniqueness} holds for the \emph{L\'evy-driven stochastic differential equation} (SDE, for short) \begin{equation*}
	dX_t = \sigma(X_{t-}) \, dL_t, \qquad X_0 \sim \mu,
\end{equation*}
if any two weak solutions of the SDE have the same finite-dimensional distributions. We refer the reader to Ikeda \& Watanabe \cite{ikeda} and Situ \cite{situ} for further details. \par
Let $(A,\mc{D})$ be a linear operator with domain $\mc{D} \subseteq \mc{B}_b(\mbb{R}^d)$ and $\mu$ a probability measure on $(\mbb{R}^d,\mc{B}(\mbb{R}^d))$. A $d$-dimensional stochastic process $(X_t)_{t \geq 0}$ with \cadlag sample paths is a \emph{solution to the $(A,\mc{D})$-martingale problem with initial distribution $\mu$} if $X_0 \sim \mu$ and \begin{equation*}
	M_t^f := f(X_t)-f(X_0)- \int_0^t Af(X_s) \, ds, \qquad t \geq 0,
\end{equation*}
is a martingale with respect to the canonical filtration of $(X_t)_{t \geq 0}$ for any $f \in \mc{D}$. The $(A,\mc{D})$-martingale problem is \emph{well-posed} if for any initial distribution $\mu$ there exists a unique (in the sense of finite-dimensional distributions) solution to the $(A,\mc{D})$-martingale problem with initial distribution $\mu$. For a comprehensive study of martingale problems see \cite[Chapter 4]{ethier}.

\section{Conservativeness of Feller processes and solutions to martingale problems} \label{cons}

In this section we establish a sufficient condition for the conservativeness of a class of stochastic processes, including Feller processes and solutions of martingale problems. It will be used in the next section to prove that the time-change of a conservative Feller process is conservative, see Theorem~\ref{time-5} for the precise statement We start with the following auxiliary result.

\begin{lem} \label{time-7} 
	 Let $(X_t)_{t \geq 0}$ be an $\mbb{R}^d_{\Delta}$-valued stochastic process with \cadlag sample paths and $x \in \mbb{R}^d$ such that \begin{equation}
		\mbb{E}^x u(X_{t \wedge \tau_r^x}) -u(x) = \mbb{E}^x \left( \int_{(0,t \wedge \tau_r^x)} Au(X_s) \, ds \right), \qquad t \geq 0, r>0 \label{dynkin}
	\end{equation}
	for all $u \in C_c^{\infty}(\mbb{R}^d)$ where $\tau_r^x :=\inf\{t \geq 0; |X_t-x|>r\}$ denotes the first exit time from the closed ball $\overline{B(x,r)}$ and \begin{equation*}
		Au(z) := - \int_{\mbb{R}^d} p(z,\xi) e^{iz \cdot \xi} \hat{u}(\xi) \, d\xi, \qquad z \in \mbb{R}^d
	\end{equation*}
	for a family of continuous negative definite functions $(p(z,\cdot))_{z \in \mbb{R}^d}$. Suppose that $p(\cdot,0)=0$ and that for any compact set $K \subseteq \mbb{R}^d$ there exists a constant $c>0$ such that $|p(z,\xi)| \leq c (1+|\xi|^2)$ for all $z \in K$, $\xi \in \mbb{R}^d$. If \begin{equation*}
		\liminf_{r \to \infty} \sup_{|z-x| \leq 2r} \sup_{|\xi| \leq r^{-1}} |p(z,\xi)|<\infty,
	\end{equation*}
	then \begin{equation*}
		\lim_{r \to \infty} \mbb{P}^x \left( \sup_{s \leq t} |X_s| \geq r \right) = 0 \fa t \geq 0.
	\end{equation*}
\end{lem}

\begin{proof}
	The first part of the proof is similar to the proof of the maximal inequality for Feller processes, cf.\ \cite[Corollary 5.2]{ltp} or \cite[Lemma 1.29]{diss,matters}. Pick $\chi \in C_c^{\infty}(\mbb{R}^d)$ such that $\spt \chi \subseteq B(0,1)$ and $0 \leq \chi \leq 1 = \chi(1)$. Set $\chi_r^x(z) := \chi((z-x)/r)$ for fixed $x \in \mbb{R}^d$, $r>0$. Since $\chi_r^x \in C_c^{\infty}(\mbb{R}^d)$, we have \begin{equation*}
		1-\mbb{E}^x(\chi_r^x (X_{t \wedge \tau_r^x})) = \mbb{E}^x \left( \int_{(0,\tau_r^x)} A\chi_r^x(X_{s}) \, ds \right).
	\end{equation*}
	Using that \begin{equation*}
		\mbb{P}^x \left( \sup_{s \leq t} |X_s-x| > r \right) 
		\leq \mbb{P}^x(\tau_r^x \leq t)
		\leq \mbb{E}^x(1-\chi_r^x(X_{t \wedge \tau_r^x}))
	\end{equation*}
	and \begin{align*}
		A \chi_r^x(z) 
		= - \int_{\mbb{R}^d} p(z,\xi) \hat{\chi}_r^x(\xi) e^{iz \cdot \xi} \,d\xi 
		&=- \int_{\mbb{R}^d} p(z,\xi) r^d \hat{\chi}(r \xi) e^{i(z-x) \cdot \xi} \,d\xi \\
		&=- \int_{\mbb{R}^d} p(z,\xi/r) \hat{\chi}(\xi) e^{i (z-x) \cdot \xi/r} \, d\xi
	\end{align*}
	we find \begin{align}
		\mbb{P}^x \left( \sup_{s \leq t} |X_s-x| > r \right)
		&\leq \left| \mbb{E}^x \left( \int_{(0,\tau_r^x)} \left[ \I_{|z-x| \leq r} \int_{\mbb{R}^d} p(z,\xi/r) \hat{\chi}(\xi) e^{i(z-x) \cdot \xi/r} \, d\xi \right] \bigg|_{z=X_{s}}  \, ds \right) \right| \notag \\
		&\leq \mbb{E}^x \left( \int_0^{t} \left[ \I_{|z-x| \leq r} \int_{\mbb{R}^d} |p(z,\xi/r)| \cdot |\hat{\chi}(\xi)| \, d\xi \right] \bigg|_{z=X_{s}} \, ds \right). \label{time-eq85}
	\end{align}
	Pick a cut-off function $\kappa \in C_c^{\infty}(\mbb{R}^d)$ such that $\I_{B(0,1)} \leq \kappa \leq \I_{B(0,2)}$. If we set \begin{equation*}
		g_r(z) :=  \kappa((z-x)/r) \int_{\mbb{R}^d} |p(z,\xi/r)| \cdot |\hat{\chi}(\xi)| \, d\xi 
	\end{equation*}
	then the above estimate shows
	\begin{align*}
		\mbb{P}^{x} \left( \sup_{s \leq t} |X_s-x| >r \right)
		\leq \int_0^t \mbb{E}^{x} g_r(X_{s}) \, ds.
	\end{align*}
	 As $\sup_{|z| \leq 2r} |p(z,\xi)| \leq c (1+|\xi|^2)$ an application of the dominated convergence theorem gives $g_r(z) \xrightarrow[]{r \to \infty} 0$ for all $z \in \mbb{R}^d$. Since  \begin{equation*}
		 g_r(z) \leq c' \sup_{|y-x| \leq 2r} \sup_{|\eta| \leq r^{-1}} |p(y,\eta)|  \int_{\mbb{R}^d}  (1+|\xi|^2) |\hat{\chi}(\xi)| \, d\xi,
	\end{equation*}
	cf.\ \cite[Proposition 2.17d)]{ltp}, there exists, by assumption, a sequence $(r_k)_{k \in \mbb{N}} \subseteq (0,\infty)$ such that $r_k \to \infty$ and $\sup_k \sup_z g_{r_k}(z)<\infty$. Applying the dominated convergence theorem yields  \begin{equation*}
		\lim_{k \to \infty} \mbb{P}^{x} \left( \sup_{s \leq t} |X_s-x| > r_k \right)=0. \qedhere
	\end{equation*}
\end{proof}

Lemma~\ref{time-7} applies, in particular, to solutions of martingale problems.

\begin{kor} \label{cons-3}
	Let $A$ be a pseudo-differential operator with continuous negative definite symbol $p$, i.\,e.\ \begin{equation*}
		Af(x) = - \int_{\mbb{R}^d} e^{ix \cdot \xi} p(x,\xi) \hat{f}(\xi) \, d\xi, \qquad f \in C_c^{\infty}(\mbb{R}^d), x \in \mbb{R}^d.
	\end{equation*}
	Suppose that $p(\cdot,0)=0$ and that for any compact set $K \subseteq \mbb{R}^d$ there exists a constant $c>0$ such that $|p(z,\xi)| \leq c (1+|\xi|^2)$ for all $z \in K$, $\xi \in \mbb{R}^d$. Let $(X_t)_{t \geq 0}$ be an $\mbb{R}^d_{\Delta}$-valued solution to the $(A,C_c^{\infty}(\mbb{R}^d))$-martingale problem with initial distribution $\mu=\delta_x$. If \begin{equation*}
			\liminf_{r \to \infty} \sup_{|z-x| \leq 2r} \sup_{|\xi| \leq r^{-1}} |p(z,\xi)|<\infty,
		\end{equation*}
		then $(X_t)_{t \geq 0}$ is conservative.
\end{kor}

For Feller processes a slightly stronger statement holds true:

\begin{lem} \label{time-8} 
	Let $(X_t)_{t \geq 0}$ be a rich Feller process with generator $(A,\mc{D}(A))$ and symbol $p$ such that $p(\cdot,0)=0$. Suppose that there exists a set $U \subseteq \mbb{R}^d$ such that \begin{equation}
		\liminf_{r \to \infty}  \sup_{|z-x| \leq 2r} \sup_{|\xi| \leq r^{-1}} |p(z,\xi)|<\infty \fa x \in U. \label{time-eq7}
	\end{equation}
	If $(x_n)_{n \in \mbb{N}} \subseteq U$ is a sequence such that $x_n \to x \in U$, then $(X_t)_{t \geq 0}$ satisfies the \emph{compact containment condition} \begin{equation*}
		\lim_{r \to \infty} \sup_{n \in \mbb{N}} \mbb{P}^{x_n} \left( \sup_{s \leq t} |X_s| > r \right) = 0 \fa t \geq 0.
	\end{equation*}
	In particular, if $U=\mbb{R}^d$, then $(X_t)_{t \geq 0}$ is conservative.
\end{lem}

For the particular case that $x_n := x$ we recover a result by Wang \cite[Theorem 2.1]{wang} which states that a rich Feller process with symbol $p$ is conservative if \begin{equation*}
	\liminf_{r \to \infty} \sup_{|z-x| \leq r} \sup_{|\xi| \leq r^{-1}} |p(z,\xi)| < \infty \fa x \in \mbb{R}^d.
\end{equation*}
Let us remark that the proof of Lemma~\ref{time-8} becomes much easier if we replace \eqref{time-eq7} by the stronger assumption \begin{equation}
	\liminf_{r \to \infty} \sup_{|z-x| \leq 2r} \sup_{|\xi| \leq r^{-1}} |p(z,\xi)| = 0 \fa x \in \mbb{R}^d; \label{time-eq8}
\end{equation}
in this case Lemma~\ref{time-8} is a direct consequence of the maximal inequality which states that \begin{equation*}
	\mbb{P}^x \left( \sup_{s \leq t} |X_s-x| > r \right) \leq c t \sup_{|z-x| \leq r} \sup_{|\xi| \leq r^{-1}} |p(z,\xi)| , \qquad x \in \mbb{R}^d,
\end{equation*}
for an absolute constant $c>0$, cf.\ \cite[Corollary 5.2]{ltp} or \cite[Lemma 1.29]{diss,matters}. Later on, in Section~\ref{app}, we will consider solutions of SDEs driven by a one-dimensional isotropic $\alpha$-stable L\'evy process $(L_t)_{t \geq 0}$ \begin{equation*}
	dX_t = \sigma(X_{t-}) \,d L_t, \qquad X_0 = x,
\end{equation*}
the symbol of $(X_t)_{t \geq 0}$ equals $p(x,\xi) = |\sigma(x)|^{\alpha} \, |\xi|^{\alpha}$, and therefore \eqref{time-eq7} allows us to consider coefficients $\sigma$ of linear growth whereas \eqref{time-eq8} would restrict us to functions $\sigma$ of \emph{sub}linear growth. \par \medskip

\begin{proof}[Proof of Lemma~\ref{time-8}]
	Let $(X_t)_{t \geq 0}$ be a rich Feller process with symbol $p$. Then the Dynkin formula \eqref{dynkin} holds, and it follows from \cite[Theorem 2.31]{ltp} that the other assumption of Lemma~\ref{time-7} are satisfied. Let $(y_n)_{n \in \mathbb{N}} \subseteq U$ be a sequence such that $y_n \to y \in U$. Then $\overline{B(y_n,r)} \subseteq B(y,3r/2)$ for sufficiently large $r>0$. Pick a cut-off function $\kappa \in C_c^{\infty}(\mbb{R}^d)$ such that $\I_{B(0,3/2)} \leq \kappa \leq \I_{B(0,2)}$. If we set \begin{equation*}
			g_r(z):= \kappa((z-y)/r) \int_{\mbb{R}^d} |p(z,\xi/r)| \cdot |\hat{\chi}(\xi)| \, d\xi 
		\end{equation*}
		then  \eqref{time-eq85} shows
		\begin{align*}
			\mbb{P}^{y_n} \left( \sup_{s \leq t} |X_s-y_n| >r \right)
			\leq \int_0^t \mbb{E}^{y_n} g_r(X_{s}) \, ds \fa n \in \mbb{N}.
	 	\end{align*}
	 	As $p(\cdot,0)=0$, we obtain from \cite[Theorem 2.31]{ltp} that $p(\cdot,\xi)$ is continuous for all $\xi \in \mbb{R}^d$. Using that $\sup_{z \in K} |p(z,\xi)| \leq c (1+|\xi|^2)$ for any compact set $K \subseteq \mbb{R}^d$, it follows from the dominated convergence theorem that $g_b \in C_b(\mbb{R}^d)$. Since $(X_t)_{t \geq 0}$ is a conservative Feller process, $\mbb{P}_{X_t}^{y_n}$ converges weakly to $\mbb{P}_{X_t}^y$, and therefore we find \begin{align*}
		 	\limsup_{n \to \infty} \mbb{P}^{y_n} \left( \sup_{s \leq t} |X_s-y_n| >r \right)
		 	\leq \int_0^t \mbb{E}^y g_r(X_{s}) \, ds.
		 \end{align*}
		The proof of Lemma~\ref{time-7} shows that there exists a sequence $(r_k)_{k \in \mbb{N}} \subseteq (0,\infty)$ such that $r_k \to \infty$ and \begin{align*}
				\limsup_{k \to \infty} \lim_{n \to \infty} \mbb{P}^{y_n} \left( \sup_{s \leq t} |X_s-y_n| > r_k \right) \leq \lim_{k \to \infty} \int_0^t \mbb{E}^y g_{r_k}(X_{s}) \, ds=0.
			\end{align*}
			Using the boundedness of the sequence $(y_n)_{n \in \mbb{N}}$, the assertion follows.
\end{proof}

\section{Time changes of Feller processes} \label{time}

In this section we are interested in Feller processes with symbols of the form \begin{equation*}
	p(x,\xi) = \varphi(x) q(x,\xi)
\end{equation*}
where $\varphi: \mbb{R}^d \to (0,\infty)$ is a continuous function and $q$ the symbol of a rich Feller process. For bounded functions $\varphi \in C_b(\mbb{R}^d)$ it is well known that a Feller process with symbol $p$ can be obtained by a random time change, cf.\ \cite[Corollary 4.2]{ltp} and \cite[Theorem 2]{lumer}.

\begin{thm} \label{time-3}
	Let $(X_t)_{t \geq 0}$ be a conservative rich Feller process with generator $(A,\mc{D}(A))$ and symbol $q$. For a continuous bounded mapping $\varphi: \mbb{R}^d \to (0,\infty)$ denote by $\alpha_t(\omega)$ the unique number such that \begin{equation*}
		t = \int_{0}^{\alpha_t(\omega)} \frac{1}{\varphi(X_s(\omega))} \, ds.
	\end{equation*}
	Then $Y_t := X_{\alpha_t}$, $t \geq 0$, defines a rich Feller process with symbol $p(x,\xi) = \varphi(x) q(x,\xi)$. Moreover, the infinitesimal generator of $(Y_t)_{t \geq 0}$ is the closure of $(\varphi(\cdot)A, \mc{D}(A))$; in particular any core for the generator $(A,\mc{D}(A))$ is a core for the infinitesimal generator of $(Y_t)_{t \geq 0}$.
\end{thm}

Note that $\alpha_t$ is well-defined since, by the boundedness of $\varphi$, \begin{equation*}
	\int_0^{u} \frac{1}{\varphi(X_s(\omega))} \, ds \geq \frac{u}{\|\varphi\|_{\infty}} \fa u \geq 0;
\end{equation*}
in particular $(Y_t)_{t \geq 0}$ does not explode in finite time. Time changes for bounded functions $\varphi$ which need not to be continuous have been studied by Kr\"{u}hner \& Schnurr \cite{kruehner}. \par \medskip

We will establish sufficient conditions on the symbol $q(x,\xi)$ which ensure the existence of Feller processes with symbol $p(x,\xi) = \varphi(x) q(x,\xi)$ for \emph{un}bounded functions $\varphi$, cf.\ Theorem~\ref{time-13}. For the particular case that $q(x,\xi)=\psi(\xi)$ does not depend on $x$, i.\,e.\ $q=\psi$ is the characteristic exponent of a L\'evy process, there are two existence results by Kolokoltsov \cite{kol04,kol}; in \cite{kol} it is assumed that $\varphi$ is twice differentiable and that the associated L\'evy measure $\nu$ has a finite second moment. The result in \cite{kol04} requires that $\varphi$ is bounded whenever the L\'evy measure $\nu$ does not have bounded support. Both results are quite restrictive; for instance they do not allow us to construct Feller processes with symbol $p(x,\xi) = \varphi(x) |\xi|^{\alpha}$, $\alpha \in (0,2)$, for unbounded functions $\varphi$. \par
We will combine the random time change technique with a classical perturbation result for Feller semigroups to obtain a new existence result for Feller processes with decomposable symbols. The first step is to investigate whether the random time change of a Feller process is a $C_b$-Feller process.

\begin{thm} \label{time-5} 
Let $(X_t)_{t \geq 0}$ be a rich Feller process with symbol $q$ and generator $(A,\mc{D}(A))$ such that $q(x,0)=0$ for all $x \in \mbb{R}^d$. Let $\varphi: \mbb{R}^d \to (0,\infty)$ be a continuous mapping such that \begin{equation}
	\liminf_{R \to \infty} \sup_{|y| \leq 4R} \sup_{|\xi| \leq R^{-1}} \max\{\varphi(y),1\} \cdot |q(y,\xi)| < \infty. \label{time-eq5}
\end{equation}
(Note that \eqref{time-eq5} implies, by Lemma~\ref{time-8}, that $(X_t)_{t \geq 0}$ is conservative.) Set \begin{equation*}
	r_n(\omega) := \int_{(0,n)} \frac{1}{\varphi(X_s(\omega))} \, ds, \qquad n \in \mbb{N} \cup \{\infty\},
\end{equation*}
and for $t<r_{\infty}(\omega)$ denote by $\alpha_t(\omega)$ the unique number such that $t = \int_0^{\alpha_t(\omega)} 1/\varphi(X_s(\omega)) \, ds$.  Then the process  $(Y_t)_{t \geq 0}$ defined by \begin{equation*}
	Y_t(\omega) := \begin{cases} X_{\alpha_t}(\omega), & t<r_{\infty}(\omega), \\ \Delta,  &t \geq r_{\infty}(\omega) \end{cases}
\end{equation*}
is a conservative $C_b$-Feller process; in particular, $\mathbb{P}^x(r_{\infty}<\infty)=0$ for all $x \in \mbb{R}^d$.
\end{thm}

There are several results in the literature which give sufficient conditions which ensure that the random time change of a $C_b$-Feller process $(X_t)_{t \geq 0}$ is a $C_b$-Feller process; typically, they assume that $(X_t)_{t \geq 0}$ is uniformly stochastically continuous, i.\,e.\ \begin{equation*}
	\lim_{t \to 0} \sup_{x \in \mbb{R}^d} \mbb{P}^x \left( \sup_{s \leq t} |X_s-x| > \delta\right) = 0 \fa \delta>0,
\end{equation*}
see e.\,g.\ Lamperti \cite{lamperti} or Helland \cite{helland}. This condition fails, in general, to hold for Feller processes with unbounded coefficients, and therefore it is too restrictive for our purpose. \par
Let us also mention that it is, in general, hard to verify whether a time-changed process is conservative, i.\,e.\ whether \begin{equation*}
	\int_{(0,\infty)} \frac{1}{\varphi(X_s)} \, ds = \infty \quad \text{a.s.};
\end{equation*}
as we will see in the proof of Theorem~\ref{time-5} the growth condition \eqref{time-eq5} is a sufficient condition for the conservativeness. Since the result is of independent interest we formulate it as a corollary.

\begin{kor} \label{time-6}
	Let $(X_t)_{t \geq 0}$ be a rich Feller process with symbol $q$ such that $q(\cdot,0)=0$. If $f: \mbb{R}^d \to (0,\infty)$ is a continuous mapping and 
	\begin{equation*}
		\liminf_{R \to \infty} \sup_{|y| \leq 4R} \sup_{|\xi| \leq R^{-1}} \max\left\{\frac{1}{f(y)},1 \right\} \, |q(y,\xi)| < \infty,
	\end{equation*}
	then \begin{equation}
		\int_{(0,\infty)} f(X_s) \, ds = \infty \qquad \text{$\mbb{P}^x$-a.s. for all $x \in \mbb{R}^d$}. \label{time-eq55}
	\end{equation}
\end{kor}

Corollary~\ref{time-6} applies, in particular, if $(X_t)_{t \geq 0}$ is a L\'evy process with characteristic exponent $\psi$; in this case the growth condition reads 
\begin{equation}
	\liminf_{R \to \infty} \left( \sup_{|y| \leq 4R} \frac{1}{f(y)} \sup_{|\xi| \leq R^{-1}} |\psi(\xi)| \right) < \infty. \label{time-eq6}
\end{equation}
For instance if $(X_t)_{t \geq 0}$ is an isotropic $\alpha$-stable L\'evy process, $\alpha \in (0,2]$, then \eqref{time-eq55} holds for any continuous function $f: \mbb{R}^d \to (0,\infty)$ such that $|f(x)| \geq c/(1+|x|^{\alpha})$, $x \in \mbb{R}^d$. We would to point out that the dimension $d$ plays an important role for the (in)finiteness of the perpetual integral; for instance the one-dimensional isotropic $\alpha$-stable L\'evy process, $\alpha \in (1,2)$, is recurrent, and therefore \begin{equation*}
	\int_{(0,\infty)} f(X_s) \, ds = \infty \quad \text{a.s.}
\end{equation*}
is trivially satisfied for any function $f>0$. It is far from being obvious how this result can be generalized to higher dimensions since the isotropic $\alpha$-stable L\'evy processis transient in dimension $d \geq 2$; in particular, we cannot expect the perpetual integral to be infinite almost surely without additional growth assumptions on $f$. Note that our result, Corollary~\ref{time-6}, applies in \emph{any} dimension $d \geq 1$. \par
Perpetual integrals $\int_{(0,\infty)} f(X_s) \, ds$ for one-dimensional L\'evy processes $(X_t)_{t \geq 0}$ have been studied by D\"{o}ring \& Kyprianou \cite{doering}, but their result requires that $(X_t)_{t \geq 0}$ has local times and finite mean; e.\,g.\ for an isotropic $\alpha$-stable L\'evy process this means that $\alpha \in (1,2]$ (the almost sure explosion of the perpetual integral is then trivial because the isotropic $\alpha$-stable process is recurrent in dimension $d=1$).  \par \medskip

\begin{proof}[Proof of Theorem~\ref{time-5}]
	The first step is to prove that $(Y_t)_{t \geq 0}$ is conservative. By \eqref{time-eq5} and Lemma~\ref{time-7} it suffices to show that  \begin{equation*}
		\mbb{E}^x u(Y_{t \wedge \tau_r^x})-u(x) = \mbb{E}^x \left( \int_{(0,t \wedge \tau_r^x)} \varphi(Y_s) Au(Y_s) \, ds \right), \qquad x \in \mbb{R}^d, t \geq 0,
	\end{equation*}
	for all $u \in C_c^{\infty}(\mbb{R}^d)$; as usual $\tau_r^x := \inf\{t \geq 0; |Y_t-x| > r\}$ denotes the first exit time from $\overline{B(x,r)}$. Fix $u \in C_c^{\infty}(\mbb{R}^d)$, and let $(\mc{F}_t)_{t \geq 0}$ be an admissible right-continuous filtration for $(X_t)_{t \geq 0}$, see \cite[Theorem 1.20]{ltp} for one possible choice. Since $(X_t)_{t \geq 0}$ is a rich Feller process, there exists a martingale $(M_t)_{t \geq 0}$ with respect to $(\mc{F}_{t})_{t \geq 0}$ such that \begin{equation*}
		u(X_t)-u(x)-M_t = \int_0^t Au(X_s) \, ds;
	\end{equation*}
	By the very definition of the time change, this implies \begin{align*}
		u(X_{\alpha(t) \wedge n})-u(x)-M_{\alpha(t) \wedge n} = \int_{(0,t \wedge r_n)} \varphi(Y_s) Au(Y_s) \,ds;
	\end{align*}
	see \cite[proof of Corollary 4.2]{ltp} for details (recall that $r_n := \int_0^n 1/\varphi(X_s) \, ds$). For $n \in \mbb{N} \cup \{\infty\}$ define \begin{equation*}
		\sigma^{(n)} := \inf\left\{t \geq 0; \sup_{s \leq \alpha_t \wedge n} |X_s-x|> r \right\};
	\end{equation*}
	note that the continuity of $t \mapsto \alpha_t$ implies $\sigma^{(\infty)} = \tau_r^x$. By the optional stopping theorem, $(M_{\alpha(t) \wedge n},\mc{F}_{\alpha(t) \wedge n})_{t \geq 0}$ is a martingale. Since $\sigma^{(n)}$ is an $\mc{F}_{\alpha(t) \wedge n}$-stopping time, another application of the optional stopping theorem yields \begin{align*}
		\mbb{E}^x u(X_{\alpha(\sigma^{(n)} \wedge t) \wedge n})-u(x) = \mbb{E}^x \left( \int_{(0,\sigma^{(n)} \wedge t \wedge r(n))} \varphi(Y_s) Au(Y_s) \, ds \right).
	\end{align*}
	It is not difficult to see that $\sigma^{(n)} \downarrow \sigma^{(\infty)}=\tau_r^x$ as $n \to \infty$. 
	Hence, by the dominated convergence theorem,
	\begin{equation*}
		\mbb{E}^x u(Y_{t \wedge \tau_r^x})-u(x)= \mbb{E}^x \left( \int_{(0,t \wedge \tau_r^x)} \varphi(Y_s) Au(Y_s) \, ds \right)
	\end{equation*}
	where we use the convention that $f(\Delta):=0$ for $f:\mbb{R}^d \to \mbb{R}$. This shows that \eqref{dynkin} holds with $p(x,\xi):=\varphi(x) q(x,\xi)$. Applying Lemma~\ref{time-7} we find that $(Y_t)_{t \geq 0}$ is conservative. \par
	It remains to show that $(Y_t)_{t \geq 0}$ is a $C_b$-Feller process. It is known that $(Y_t)_{t \geq 0}$ is a strong Markov process, see e.\,g.\ \cite{volk}, and therefore it suffices to prove the weak continuity: \begin{equation}
		\mbb{P}^{x_n}_{Y_t} \xrightarrow[n \to \infty]{\text{weakly}} \mbb{P}^x_{Y_t} \fa \text{sequences $x_n \to x$, $t \geq 0$}. \label{time-eq9}
	\end{equation}
	 In the remaining part of the proof we use the canonical model, i.\,e.\ we consider $(X_t)_{t \geq 0}$ and $(Y_t)_{t \geq 0}$ as mappings $X: D([0,\infty), \mbb{R}^d) \to \mathbb{R}^d$ and $Y: D([0,\infty), \mbb{R}^d_{\Delta}) \to \mbb{R}^d_{\Delta}$ where $D([0,\infty),E)$ denotes the space of \cadlag functions $\omega: [0,\infty) \to E$. If we define \begin{equation*}
		f: D([0,\infty), \mbb{R}^d) \to D([0,\infty),\mbb{R}^d_{\Delta}), \quad \omega \mapsto f(\omega)(t) := \begin{cases} \omega(\alpha_t(\omega)), & t < r_{\infty}(\omega), \\ \Delta, & t \geq r_{\infty}(\omega), \end{cases}
	\end{equation*}
	then $Y_t = f(X)(t)$. In order to prove \eqref{time-eq9}, we fix a sequence $x_n \to x$ and denote by $X^{(n)}$ the process started at $x_n$ and by $X^{(0)}$ the Feller process started at $x$.  For each $n \in \mbb{N}_0$ the process $X^{(n)}$ induces a probability measure $\mbb{P}^{(n)}$ on $D([0,\infty),\mbb{R}^d)$. Clearly, \eqref{time-eq9} is equivalent to \begin{equation}
		f(X^{(n)})(t) \xrightarrow[n \to \infty]{d} f(X^{(0)})(t). \label{time-eq11}
	\end{equation}
	Since $(X_t)_{t \geq 0}$ is a Feller process, we have $X^{(n)}(t) \xrightarrow[]{d} X^{(0)}(t)$ for all $t \geq 0$, and by the Markov property this implies $X^{(n)} \to X^{(0)}$ in finite-dimensional distribution. On the other hand, Lemma~\ref{time-8} shows \begin{align*}
		\sup_{n \in \mbb{N}_0} \mbb{P}^{(n)} \left( \sup_{s \leq t} |X_s^{(n)}| > R \right)
		\xrightarrow[R \to \infty]{} 0.
	\end{align*}
	It follows from \cite[Theorem 4.9.2]{kol} that $(X^{(n)})_{n \in \mbb{N}_0}$ is tight, and this, in turn, implies by Prohorov's theorem, cf.\ \cite[Theorem 2.2, p.~104]{ethier}, relative compactness in $D([0,\infty), \mbb{R}^d)$.  Applying \cite[Theorem 7.8,  p.~131]{ethier} we get $X^{(n)} \to X^{(0)}$ in $D([0,\infty), \mbb{R}^d)$. Since $f$ is $\mbb{P}^{(0)}$-a.s.\ continuous, cf.\ \cite[Theorem 2.7]{helland}, the continuous mapping theorem yields \begin{equation*} 
		f(X^{(n)}) \xrightarrow[]{d} f(X^{(0)}).
	\end{equation*} 
	As $X$ is quasi-leftcontinuous, see \cite[p.~127]{ito}, and $\alpha_t$ is a predictable stopping time, we have \begin{equation*}
		\mbb{P}^{(0)}\big(\{f(X^{(0)})(t)=f(X^{(0)})(t-), t<r_{\infty}(X^{(0)})\} \big)=1
	\end{equation*}
	for fixed $t>0$. Since we already know that $(Y_t)_{t \geq 0}$ is conservative, i.\,e.\ $\mbb{P}^{(0)}(r_{\infty}=\infty)=1$, we find that the mapping $s \mapsto f(X^{(0)})(s)$ is $\mbb{P}^{(0)}$-a.s.\ continuous at $s=t$. This means that the projection $y \mapsto y(t)$ is $\mbb{P}^{(0)}$-a.s.\ continuous at $y=f(X^{(0)})$. Applying the continuous mapping theorem another time, we conclude \begin{equation*}
		f(X^{(n)})(t) \xrightarrow[]{d} f(X^{(0)})(t) \fa t \geq 0. \qedhere
	\end{equation*}
\end{proof}

Combining Theorem~\ref{time-5} with the following perturbation result will allow us to construct Feller processes with decomposable symbols.

\begin{lem} \label{time-9}
	Let $(q(x,\cdot))_{x \in \mbb{R}^d}$ be a family of continuous negative definite functions with characteristics $(b,Q,\nu)$ such that $q(\cdot,\xi)$ is continuous for all $\xi \in \mbb{R}^d$ and $q(x,0)=0$ for all $x \in \mbb{R}^d$. Let $R: \mbb{R}^d \to [1,\infty)$ and $\chi: \mbb{R}^d \to [0,1]$ be continuous mappings such that $\I_{B(0,1)} \leq \chi \leq \I_{B(0,2)}$. Set \begin{equation}
		q_R(x,\xi) := -i b(x) \cdot \xi + \frac{1}{2} \xi \cdot Q(x) \xi  + \int_{\mbb{R}^d} \chi \left( \frac{y}{R(x)} \right) (1-e^{iy \cdot \xi}+iy \cdot \xi \I_{(0,1)}(|y|)) \, \nu(x,dy) \label{time-eq13}
	\end{equation} 
	Suppose that \begin{equation}
		\sup_{x \in \mbb{R}^d} \int_{|y|>R(x)} \nu(x,dy) < \infty \quad \text{and} \quad \lim_{|x| \to \infty} \nu(x,B(-x,r)) = 0 \fa r>0. \label{time-eq15}
	\end{equation}
	 \begin{enumerate}
		\item If there exists a rich Feller process with symbol $q$ and $C_c^{\infty}(\mbb{R}^d)$ is a core for its generator, then there exists a rich Feller process with symbol $q_R$.
		\item If there exists a rich Feller process with symbol $q_R$, then there exists a rich Feller process with symbol $q$.
	\end{enumerate}
\end{lem}

\begin{bem} \label{time-11}
	 Under the assumption of Lemma~\ref{time-9} the truncated symbol $q_R$ is continuous with respect to the space variable $x$. Let us consider both cases separately: \begin{enumerate}
		 \item If there exists a rich Feller process with symbol $q$, then \cite[Theorem 4.4]{rs98} shows that $x \mapsto q(x,\xi)=q(x,\xi)-q(x,0)$ is continuous for all $\xi \in \mbb{R}^d$. Since $q$ is locally bounded, cf.\ \cite[Proposition 2.27(d)]{ltp}, it follows from Corollary~\ref{appen-3} in the appendix that $x \mapsto q_R(x,\xi)$ is continuous for all $\xi \in \mbb{R}^d$.
		 \item If there exists a rich Feller process with symbol $q_R$, then we obtain from \cite[Theorem 4.4]{rs98} that $x \mapsto q_R(x,\xi)$ is continuous for all $\xi \in \mbb{R}^d$ as $q_R(x,0)=0$.
		\end{enumerate}
\end{bem}

\begin{proof}[Proof of Lemma~\ref{time-9}]
	Set \begin{equation*}
		Pf(x) := \int_{\mbb{R}^d} \left(1- \chi \left[ \frac{y}{R(x)} \right] \right) (f(x+y)-f(x)) \, \nu(x,dy), \qquad x \in \mbb{R}^d.
	\end{equation*}
	As \begin{equation*}
		\sup_{x \in \mbb{R}^d} |Pf(x)| \leq 2 \|f\|_{\infty} \sup_{x \in \mbb{R}^d} \int_{|y|>R(x)} \, \nu(x,dy) < \infty
	\end{equation*}
	we see that $P$ is a bounded linear operator. By combining the dominated convergence theorem with Remark~\ref{time-11} we find that \begin{equation*}
		x \mapsto Pf(x) = - \int_{\mbb{R}^d} (q(x,\xi)-q_R(x,\xi)) e^{ix \cdot \xi} \hat{f}(\xi) \, d\xi
	\end{equation*}
	is continuous for any $f \in C_c^{\infty}(\mbb{R}^d)$. If $f \in C_c^{\infty}(\mbb{R}^d)$ is compactly supported in $B(0,r)$ for some $r>0$, then \begin{equation*}
		|Pf(x)| \leq 2 \|f\|_{\infty} \nu(x,B(-x,r))  \fa |x| \geq r
	\end{equation*}
	which implies, by \eqref{time-eq15}, $\lim_{|x| \to \infty} Pf(x)=0$. Hence, $Pf \in C_{\infty}(\mbb{R}^d)$ for all $f \in C_c^{\infty}(\mbb{R}^d)$. Using the boundedness of $P$ and that $(C_c^{\infty}(\mbb{R}^d),\|\cdot\|_{\infty})$ is dense in $(C_{\infty}(\mbb{R}^d),\|\cdot\|_{\infty})$, we get $Pf \in C_{\infty}(\mbb{R}^d)$ for all $f \in C_{\infty}(\mbb{R}^d)$, and so $P: C_{\infty}(\mbb{R}^d) \to C_{\infty}(\mbb{R}^d)$ is a bounded operator. \par
	We prove (i). Suppose that there exists a rich Feller process with symbol $q$ and that $C_c^{\infty}(\mbb{R}^d)$ is a core for the generator $(A,\mc{D}(A))$. Then \begin{equation*}
		Lu(x) := (A-P)u(x) = - \int_{\mbb{R}^d} q_R(x,\xi) e^{i x \cdot \xi} \hat{u}(\xi) \, d\xi, \qquad u \in C_{\infty}(\mbb{R}^d), x \in \mbb{R}^d
	\end{equation*}
	and therefore $L|_{C_c^{\infty}(\mbb{R}^d)}$ is dissipative. As $C_c^{\infty}(\mbb{R}^d)$ is a core, this implies that $L|_{\mc{D}(A)}$ is dissipative. Since $P$ is bounded, it follows from a classical perturbation result that $(L,\mc{D}(L))$ is the generator of a Feller semigroup, see e.\,g.\ \cite[Corollary 3.8]{davies}, and this proves (ii). \par
	Conversely, if there exists a Feller process with symbol $q_R$ and generator $(L,\mc{D}(L))$, then $A:=L+P$ is dissipative as a sum of two dissipative operators and $A$ is, when restricted to $C_c^{\infty}(\mbb{R}^d)$, a pseudo-differential operator with symbol $q$. The assertion follows from the above-mentioned perturbation result.
\end{proof}

Using Theorem~\ref{time-5} and Lemma~\ref{time-9}, we obtain the following existence result for Feller processes with unbounded coefficients.

\begin{thm} \label{time-13}
	Let $(X_t)_{t \geq 0}$ be an $\mbb{R}^d$-valued rich Feller process with characteristics $(b,Q,\nu)$ and symbol $q$. Suppose that $q(\cdot,0)=0$ and that $C_c^{\infty}(\mbb{R}^d)$ is a core for the infinitesimal generator of $(X_t)_{t \geq 0}$. Let $\varphi: \mbb{R}^d \to (0,\infty)$ be a continuous mapping. Suppose that there exist a constant $c \in (0,1)$, a continuous mapping $R: \mbb{R}^d \to [1,\infty)$ and a continuous function $\chi: \mbb{R}^d \to [0,1]$, $\I_{B(0,1)} \leq \chi \leq \I_{B(0,2)}$, such that the following properties are satisfied. \begin{enumerate}
		\item\label{time-13-i} $\max\{\varphi(x),1\} \nu(x,B(-x,r)) \xrightarrow[]{|x| \to \infty} 0$ for all $r>0$
		\item\label{time-13-ii} $\max\{\varphi(x),1\} \left|b(x) + \int_{1<|y|<R(x)} y \, \nu(x,dy) \right| \leq C (1+|x|)$ for all $x \in \mbb{R}^d$
		\item\label{time-13-iii}$\max\{\varphi(x),1\} \left( |Q(x)| + \int_{|y| \leq R(x)} |y|^2 \, \nu(x,dy) \right) \leq C (1+|x|^2)$ for all $x \in \mbb{R}^d$
		\item\label{time-13-iv} $\sup_{x \in \mbb{R}^d} \left( \max\{\varphi(x),1\} \nu(x, B(0,R(x))^c) \right) < \infty$ for all $x \in \mbb{R}^d$
		\item\label{time-13-v} $|R(x)| \leq c (1+|x|)$ for all $x \in \mbb{R}^d$
	\end{enumerate}
	Then there exists a conservative rich Feller process $(Y_t)_{t \geq 0}$ with symbol $p(x,\xi) := \varphi(x) q(x,\xi)$.
\end{thm}


\begin{proof}
	By Lemma~\ref{time-9}, there exists a rich Feller process $(X^{(R)}_t)_{t \geq 0}$ with symbol $q_R$. Denote by $(A_R,\mc{D}(A_R))$ its generator. Using the well-known estimates \begin{equation}
			\left|1-e^{iy \xi} + iy \xi \right| \leq \frac{1}{2} |y|^2 \, |\xi|^2 \qquad \qquad  \left|1-e^{iy \xi} \right| \leq \min\{2,|y| \, |\xi|\} \label{time-eq17}
	\end{equation}
	it follows that
	\begin{equation*}
		|q_R(x,\xi)| \leq |b(x)| \, |\xi| + \frac{1}{2} |Q(x)| \, |\xi|^2 + \frac{|\xi|^2}{2} \int_{|y| \leq 1} |y|^2 \, \nu(x,dy)  + |\xi| \int_{1<|y|<R(x)} |y| \, \nu(x,dy)
	\end{equation*}
	which implies, by \ref{time-13-ii} and \ref{time-13-iii}, that $q_R$ satisfies \eqref{time-eq5}. Applying Theorem~\ref{time-5}, we find that the time-changed process $(Y^{(R)}_t)_{t \geq 0}$ is a conservative $C_b$-Feller process. We claim that $(Y_t^{(R)})_{t \geq 0}$ is a rich Feller process with symbol $p_R(x,\xi) = \varphi(x) q_R(x,\xi)$. By Dynkin's formula (for $(X_t^{(R)})_{t \geq 0}$) and the very definition of the time-changed process $(Y_t^{(R)})_{t \geq 0}$, we get \begin{align*}
		\mathbb{E}^x u(Y_t^{(R)})-u(x) = \mathbb{E}^x \left( \int_0^t \varphi(Y_s^{(R)}) A_R u(Y_s^{(R)}) \, ds \right)
	\end{align*}
	for any $u \in \mc{D}(A_R)$ such that $\varphi \cdot A_R u$ is bounded, see e.\,g.\ \cite[Proof of Corollary 4.2]{ltp} for details. Since, by \ref{time-13-i}, $\varphi A_R u \in C_{\infty}(\mbb{R}^d)$ for any $u \in C_c^{\infty}(\mbb{R}^d) \subseteq \mc{D}(A_R)$, this implies \begin{equation*}
		\lim_{t \to 0} \frac{\mbb{E}^x u(Y_t^{(R)})-u(x)}{t} = \varphi(x) A_R u(x) \in C_{\infty}(\mbb{R}^d)
	\end{equation*}
	which shows that $C_c^{\infty}(\mbb{R}^d)$ is contained in the domain of the generator of $(Y_t^{(R)})_{t \geq 0}$ and that the generator restricted to $C_c^{\infty}(\mbb{R}^d)$ is a pseudo-differential operator with symbol $-p_R(x,\xi) =- \varphi(x) q_R(x,\xi)$. It remains to show that $(Y_t^{(R)})_{t \geq 0}$ is a Feller process. Since we know that $(Y_t^{(R))})_{t \geq 0}$ is a $C_b$-Feller process, it suffices to prove that \begin{equation*}
		\lim_{|x| \to \infty} \mbb{P}^x (|Y_t^{(R)}| \leq k) = 0 \fa k \in \mathbb{N},
	\end{equation*}
	cf.\ \cite[Theorem 1.10]{ltp}. Using a standard truncation argument and a similar reasoning as above, it is not difficult to see that 
	\begin{equation*}
		\mathbb{E}^x u(Y_t^{(R)})-u(x) = \mathbb{E}^x \left( \int_0^t \varphi(Y_s^{(R)}) A_R u(Y_s^{(R)}) \, ds \right)
	\end{equation*}
	for $u(x) := 1/(|x|^2+1)$. Since \ref{time-13-ii}, \ref{time-13-iii} and \ref{time-13-v} imply, by Taylor's formula, that \begin{equation*}
		|\varphi(x) A_R u(x)| \leq C' u(x), \qquad x \in \mbb{R}^d, 
	\end{equation*}
	for some absolute constant $C'>0$, we get \begin{align*}
		\mbb{E}^x u(Y_t^{(R)}) \leq u(x) + C' \int_0^t \mbb{E}^x u(Y_s^{(R)}) \,ds.
	\end{align*}
	By Gronwall's inequality, there exists a constant $C''>0$ such that \begin{equation*}
		\mbb{E}^x u(Y_t^{(R)}) \leq u(x) e^{C'' t} \fa t \geq 0, \, \, x \in \mbb{R}^d.
	\end{equation*}
	Thus, 
	 \begin{align*}
		\mbb{P}^x(|Y_t^{(R)}| \leq k)
		= \mbb{P}^x \left(u(Y_t^{(R)}) \geq 1/(k^2+1) \right)
		&\leq (k^2+1) \mbb{E}^x u(Y_t^{(R)}) \\
		&\leq (k^2+1) u(x) e^{C'' t} \\
		&\xrightarrow[]{|x| \to \infty} 0.
	\end{align*}
	This proves that $(Y_t^{(R)})_{t \geq 0}$ is a rich Feller process with symbol $p_R(x,\xi) = \varphi(x) q_R(x,\xi)$. Applying Lemma~\ref{time-9} another time, we find that there exists a rich Feller process $(Y_t)_{t \geq 0}$ with symbol $p(x,\xi)$. \par
	In order to prove that $(Y_t)_{t \geq 0}$ is conservative, it suffices to show that \begin{equation*}
		\liminf_{r \to \infty} \sup_{|x| \leq 2r} \sup_{|\xi| \leq r^{-1}} |p(x,\xi)|<\infty, \label{time-eq19}
	\end{equation*}
	see (the remark following) Lemma~\ref{time-8}. Using the estimates \eqref{time-eq17}, this follows easily from the growth assumptions \ref{time-13-ii}-\ref{time-13-iv}.
\end{proof}

For L\'evy processes we obtain the following corollary.

\begin{kor} \label{time-15}
	Let $(L_t)_{t \geq 0}$ be a $d$-dimensional L\'evy process with L\'evy triplet $(b,Q,\nu)$ and characteristic exponent $\psi$. Let $\varphi: \mbb{R}^d \to (0,\infty)$ be a continuous function which grows at most quadratically. Suppose that the following assumptions are satisfied. \begin{enumerate}
		\item If $b \neq 0$, then $\varphi(x) \leq C_1 (1+|x|)$ for some absolute constant $C_1>0$.
		\item $\max\{\varphi(x),1\} \left|\int_{1<|y|<|x|/2}y \, \nu(dy)\right| \leq C_2 (1+|x|)$ for some absolute constant $C_2>0$.
		\item $\max\{\varphi(x),1\} \int_{|y| \leq |x|/2} |y|^2 \, \nu(dy) \leq C_3 (1+|x|^2)$ for some absolute constant $C_3>0$.
		\item $\max\{\varphi(x),1\} \nu(\{y \in \mbb{R}^d; |y|>|x|/2\}) \leq C_4$ for some absolute constant $C_3>0$.
		\item $\max\{\varphi(x),1\}\nu(B(-x,r)) \xrightarrow[]{|x| \to \infty} 0$ for all $r>0$.
	\end{enumerate}
	Then there exists a conservative rich Feller process with symbol $q(x,\xi) := \varphi(x) \psi(\xi)$.
\end{kor}

In Section~\ref{app} we will show that the $(A,C_c^{\infty}(\mbb{R}^d))$-martingale problem for the generator $A$ of the Feller process $(X_t)_{t \geq 0}$ with symbol $q(x,\xi) = \varphi(x) \psi(\xi)$ is well-posed, cf.\ Proposition~\ref{app-3}.

\begin{proof}[Proof of Corollary~\ref{time-15}]
	Set $R(x) := \max\{2, |x|/2\}$ for $x \in \mbb{R}^d$, and let $\chi \in C(\mbb{R}^d)$ be a cut-off function such that $\I_{B(0,1)} \leq \chi \leq \I_{B(0,2)}$. By the dominated convergence theorem, \begin{equation*}
		x \mapsto q_R(x,\xi) :=-i b \cdot \xi + \frac{1}{2} \xi \cdot Q \xi +  \int_{y \neq 0} \chi \left(\frac{y}{R(x)} \right) \left(1-e^{iy \cdot \xi} +iy \cdot \xi \I_{(0,1)}(|y|) \right) \, \nu(dy)
	\end{equation*}
	is a continuous mapping for all $\xi \in \mbb{R}^d$. Now the assertion follows from Theorem~\ref{time-13}.
\end{proof}

For stable-dominated L\'evy processes Corollary~\ref{time-15} reads as follows:

\begin{kor} \label{time-17}
	Let $(L_t)_{t \geq 0}$ be a L\'evy process with L\'evy triplet $(b,Q,\nu)$ and characteristic exponent $\psi$. Suppose that $\nu|_{B(0,1)^c}$ is symmetric and $\nu(dy) \leq C |y|^{-d-\beta} \, dy$ on $B(0,1)^c$ for some $\beta \in (0,2]$ and $C>0$. Let $\varphi: \mbb{R}^d \to (0,\infty)$ be a continuous mapping which grows at most quadratically. Furthermore assume that \begin{enumerate}
		\item $\varphi(x) \leq c(1+|x|)$ if $b \neq 0$,
		\item $\varphi(x) \leq c(1+|x|^{\beta})$ if $\nu \neq 0$.
	\end{enumerate}
	Then there exists a conservative rich Feller process with symbol $q(x,\xi) = \varphi(x) \psi(\xi)$.
\end{kor}

\begin{proof}
	Since $\nu(dy) \leq C |y|^{-d-\beta}$ on $B(0,1)^c$ it follows easily that there exists a constant $c>0$ such that \begin{equation*}
		\int_{0<|y|<r} |y|^2 \, \nu(dy) \leq c r^{2-\beta} \qquad \quad \nu(B(0,r)^c) \leq c r^{-\beta}
	\end{equation*}
	for all $r \geq 1$ and \begin{equation*}
		\nu(B(-x,r)) \leq c |x|^{-\beta-1} \fa |x| \gg 1.
	\end{equation*}
	Applying Corollary~\ref{time-15} finishes the proof.
\end{proof}

Corollary~\ref{time-17} applies, for instance, to the following L\'evy processes: \begin{enumerate}
	\item isotropic $\alpha$-stable: $\psi(\xi) = |\xi|^{\alpha}$, $\alpha \in (0,2]$ ($\beta = \alpha$ in Corollary~\ref{time-17})
	\item relativistic $\alpha$-stable: $\psi(\xi) = (|\xi|^2+m^2)^{\alpha/2}-m^{\alpha}$, $\alpha \in (0,2)$, $m>0$ ($\beta = 2$ in Corollary~\ref{time-17})
	\item truncated L\'evy process: $\psi(\xi) = (|\xi|^2+m^2)^{\alpha/2} \cos \left( \alpha \arctan \frac{|\xi|}{m} \right)-m^{\alpha}$, $\alpha \in (0,2)$, $m>0$ ($\beta=2$ in Corollary~\ref{time-17})
	\item homographic: $\psi(\xi) = \lambda |\xi|^2/(1+\lambda |\xi|^2)$, $\lambda>0$ ($\beta=2$ in Corollary~\ref{time-17})
\end{enumerate} 
Each of the processes (ii)-(iv) is stable-dominated since the density of the associated L\'evy measure (exists and) decays exponentially. \par \medskip

In a similar fashion, we obtain an existence result for time changes of stable-like dominated processes.

\begin{kor} \label{time-19}
	Let $(X_t)_{t \geq 0}$ be an $\mbb{R}^d$-valued rich Feller process with symbol $q$, characteristics $(b,Q,\nu)$ and core $C_c^{\infty}(\mbb{R}^d)$. Suppose that $q(\cdot,0)=0$, $\nu(x,\cdot)|_{B(0,1)}$ is symmetric for all $x \in \mbb{R}^d$ and that there exists a mapping $\beta: \mbb{R}^d \to (0,2]$ such that $\nu(x,dy) \leq C |y|^{-d-\beta(x)}$ on $B(0,1)^c$. Let $\varphi: \mbb{R}^d \to (0,\infty)$ be a continuous mapping which grows at most quadratically, and assume that there exists a constant $c>0$ such that \begin{enumerate}
		\item $\varphi(x) \leq c (1+|x|)$ for all $x \in \mbb{R}^d$ such that $b(x) \neq 0$,
		\item $\varphi(x) \leq c (1+|x|^{\beta(x)})$ for all $x \in \mbb{R}^d$ such that $\nu(x,dy) \neq 0$.
	\end{enumerate}
	Then there exists a conservative rich Feller process with symbol $p(x,\xi) := \varphi(x) q(x,\xi)$.
\end{kor}

By \cite[Section 5.1]{matters}, see also \cite[Section 5.1]{diss}, the assumptions of Corollary~\ref{time-19} are satisfied for the following symbols: \begin{enumerate}
	\item isotropic stable-like: $q(x,\xi) = |\xi|^{\alpha(x)}$ for a H\"older continuous function $\alpha: \mbb{R}^d \to (0,2]$ such that $\inf_x \alpha(x)>0$ ($\beta(x) := \alpha(x)$ in Corollary~\ref{time-19}),
	\item relativistic stable-like: $q(x,\xi) = (|\xi|^2+m(x)^2)^{\alpha(x)/2}-m(x)^{\alpha(x)}$ for H\"older continuous mappings $\alpha: \mbb{R}^d \to (0,2]$, $m: \mbb{R}^d \to (0,\infty)$ which are bounded and bounded away from $0$ (choose $\beta(x) := 2$ in Corollary~\ref{time-19}).
	\item truncated-like: $q(x,\xi) =  (|\xi|^2+m(x)^2)^{\alpha(x)/2} \cos \left( \alpha(x) \arctan \frac{|\xi|}{m(x)} \right)-m(x)^{\alpha(x)}$ for H\"older continuous functions $\alpha: \mbb{R}^d \to (0,1)$, $m: \mbb{R}^d \to (0,\infty)$ such that \begin{equation*}
		0 < \inf_{x \in \mbb{R}^d} \alpha(x) \leq \sup_{x \in \mbb{R}^d} \alpha(x)<1 \quad \text{and} \quad 0 < \inf_{x \in \mbb{R}^d} m(x) \leq \sup_{x \in \mbb{R}^d} m(x) < \infty,
	\end{equation*}
	(choose $\beta(x) := 2$ in Corollary~\ref{time-19}).
\end{enumerate}
In each of the examples there is no drift part, and therefore the growth assumptions on $\varphi$ in Corollary~\ref{time-19} boil down to $\varphi(x) \leq c(1+|x|^{\beta(x)})$, $x \in \mbb{R}^d$, for some absolute constant $c>0$.

\section{Applications} \label{app}

In this section we present applications of the results obtained in Section~\ref{time}. First, we prove an uniqueness result for stochastic differential equations driven by a one-dimensional symmetric L\'evy process, cf.\ Theorem~\ref{app-1}, and in the second part we study Feller processes with decomposable symbols of the form \begin{equation*}
	q(x,\xi) = \sum_{j=1}^n \varphi_j(x) \psi_j(\xi), \qquad x,\xi \in \mbb{R}^d,
\end{equation*}
cf.\ Theorem~\ref{app-5}.

\begin{thm} \label{app-1}
Let $(L_t)_{t \geq 0}$ be a one-dimensional isotropic $\alpha$-stable L\'evy process, $\alpha \in (0,2]$. For any continuous function $\sigma: \mbb{R} \to (0,\infty)$  which grows at most linearly there exists a unique weak solution to the SDE \begin{equation}
	dX_t = \sigma(X_{t-}) \, dL_t, \qquad X_0 \sim \mu, \label{app-eq5}
\end{equation}
for any initial distribution $\mu$.  The unique solution is a conservative rich Feller process with symbol $q(x,\xi) = |\sigma(x)|^{\alpha} |\xi|^{\alpha}$, $x,\xi \in \mbb{R}$.
\end{thm}

For $\alpha \in (1,2]$ it follows from Zanzotto \cite{zan02} that the SDE has a unique weak solution; for $\alpha \in (0,1]$ the existence of a unique weak solution seems to be new.

\begin{bems} \label{app-2} \begin{enumerate}
	\item It is useful to know that the solution is a Feller process since this allows us to study distributional and path properties of the solution using tools for Feller processes. For instance, \cite[Example 5.4]{moments} shows that \begin{equation*}
		\mbb{E}^x \left( \sup_{s \leq t} |X_s-x|^{\kappa} \right) \leq c t^{\kappa/\alpha}, \qquad \kappa \in [0,\alpha), x \in \mbb{R}^d, t \leq 1
	\end{equation*}
	if $\sigma$ is of sublinear growth. Moreover if $\alpha \in (1/2,2]$ and $\sigma$ is bounded from below and H\"older continuous with H\"older exponent $\beta \in (0,1)$, $\beta>1/(2\alpha)$, then it follows from \cite[Theorem 2.12]{diss}, see also \cite[Theorem 2.14]{matters}, that for each $t>0$ the distribution of $X_t$ has a density $p_t \in L^2(\mbb{R}^d)$ with respect to Lebesgue measure. 
	\item The linear growth condition on $\sigma$ ensures that the solution $(X_t)_{t \geq 0}$ is conservative, see (the remark following) Lemma~\ref{time-8}.
	\item If $x \mapsto x + u\sigma(x)$ is non-decreasing for all $u \in (-1,1)$, then \cite[Theorem 2.1]{zheng} shows that pathwise uniqueness holds for the SDE \eqref{app-eq5}, and therefore the SDE has a unique strong solution, cf.\ \cite[Theorem 137]{situ}.
	\item Sufficient and necessary conditions for the solution of a L\'evy-driven SDE \begin{equation*}
	dX_t = \sigma(X_{t-}) \ dL_t, \qquad X_0 = x,
	\end{equation*}
	to be a Feller process were studied in \cite{sde}.  \end{enumerate} \end{bems}

For the proof of Theorem~\ref{app-1} we will use a result by Kurtz \cite{kurtz} which states that a L\'evy-driven SDE has a unique weak solution if, and only if, the associated martingale problem is well-posed. The well-posedness of the martingale problem follows from the following proposition which is of independent interest.

\begin{prop} \label{app-3}
	Let $(L_t)_{t \geq 0}$ be a L\'evy process with L\'evy triplet $(b,Q,\nu)$ and characteristic exponent $\psi$. Let $\varphi: \mbb{R}^d \to (0,\infty)$ be a continuous function which grows at most quadratically, and suppose that the assumptions (i)-(v) of Corollary~\ref{time-15} are satisfied. Then the $(A,C_c^{\infty}(\mbb{R}^d))$-martingale problem for the pseudo-differential operator $A$ with symbol $-p(x,\xi) := -\varphi(x) \psi(\xi)$ is well-posed. The unique solution is a conservative rich Feller process with symbol $p(x,\xi) = \varphi(x) \psi(\xi)$, $x,\xi \in \mbb{R}^d$.
\end{prop}

More generally, the proof of Proposition~\ref{app-3} shows that if $(X_t)_{t \geq 0}$ is a rich Feller process with symbol $q$ and $C_c^{\infty}(\mbb{R}^d)$ is a core for the generator of $(X_t)_{t \geq 0}$, then the $(A,C_c^{\infty}(\mbb{R}^d))$-martingale problem for the pseudo-differential operator $A$ with symbol $-p(x,\xi) = - \varphi(x) q(x,\xi)$ is well-posed for any function $\varphi$ satisfying the assumptions of Theorem~\ref{time-13}.

\begin{proof}[Proof of Proposition~\ref{app-3}]
	By Corollary~\ref{time-15} there exists a conservative rich Feller process $(X_t)_{t \geq 0}$ with symbol $p$. Since the generator of $(X_t)_{t \geq 0}$ is, when restricted to $C_c^{\infty}(\mbb{R}^d)$, the pseudo-differential operator $A$, it follows from Dynkin's formula that $(X_t)_{t \geq 0}$ is a solution to the $(A,C_c^{\infty}(\mbb{R}^d))$-martingale problem. It remains to show uniqueness. \par
	Pick $(\chi_k)_{k \in \mbb{N}} \subseteq C_c(\mbb{R}^d)$ such that $0 \leq \chi_k \leq 1$ and $\chi|_{B(0,k)}=1$. As $\varphi_k := \chi_k \varphi$ is bounded, there exists by Corollary~\ref{time-15} a rich Feller process $(X_t^{(k)})_{t \geq 0}$ with symbol $q_k(x,\xi) := \varphi_k(x) \psi(\xi)$. Using that $C_c^{\infty}(\mbb{R}^d)$ is a core for the generator of $(L_t)_{t \geq 0}$ and the boundedness of $\varphi_k$, we find that $C_c^{\infty}(\mbb{R}^d)$ is a core for the generator $A^{(k)}$ of $X^{(k)}$, cf.\ Theorem~\ref{time-3}. This, in turn, implies that the $(A^{(k)},C_c^{\infty}(\mbb{R}^d))$-martingale problem is well-posed, see e.\,g.\ \cite[Theorem 4.10.3]{kol} or \cite[Theorem 1.37]{diss}. Since $q(x,\xi) = q_k(x,\xi)$ for all $|x|<k$ and $\xi \in \mbb{R}^d$, it follows from \cite[Theorem 5.3]{hoh} that the $(A,C_c^{\infty}(\mbb{R}^d))$-martingale problem is well-posed.
\end{proof} 

\begin{proof}[Proof of Theorem~\ref{app-1}]
	If we set $\varphi(x) := |\sigma(x)|^{\alpha}$, then it follows from Proposition~\ref{app-3} and Corollary~\ref{time-17} that the $(A,C_c^{\infty}(\mbb{R}^d))$-martingale problem for the pseudo-differential operator $A$ with symbol $q(x,\xi) = |\sigma(x)|^{\alpha} \, |\xi|^{\alpha}$ is well-posed. By \cite{kurtz} this implies that there exists a unique weak solution to the SDE \eqref{app-eq5}.
\end{proof}

Let us formulate Proposition~\ref{app-3} for stable-dominated L\'evy processes.

\begin{kor} \label{app-4}
	Let $(L_t)_{t \geq 0}$ be a L\'evy process with L\'evy triplet $(b,Q,\nu)$ and characteristic exponent $\psi$. Suppose that $\nu|_{B(0,1)^c}$ is symmetric and that there exist constants $C>0$ and $\beta \in (0,2]$ such that \begin{equation}
		\nu(dy) \leq C |y|^{-d-\beta} \, dy \qquad \text{on $B(0,1)^c$}. \label{app-eq7}
	\end{equation}
	Let $\varphi: \mbb{R}^d \to (0,\infty)$ be a continuous mapping which grows at most quadratically. Assume that \begin{enumerate}
		\item $\varphi(x) \leq c(1+|x|)$ if $b \neq 0$,
		\item $\varphi(x) \leq c(1+|x|^{\beta})$ if $\nu \neq 0$.
	\end{enumerate}
	Then the $(A,C_c^{\infty}(\mbb{R}^d))$-martingale problem for the pseudo-differential operator $A$ with symbol $-p(x,\xi) := -\varphi(x) \psi(\xi)$ is well-posed. The unique solution is a conservative rich Feller process with symbol $p(x,\xi) = \varphi(x) \psi(\xi)$, $x,\xi \in \mbb{R}^d$.
\end{kor}

Corollary~\ref{app-4} applies, for instance, if $(L_t)_{t \geq 0}$ is isotropic $\alpha$-stable, $\alpha \in (0,2]$, (choose $\beta=\alpha$) or relativistic stable (choose $\beta=2$), see the end of Section~\ref{time} for some more examples. Note that \eqref{app-eq7} implies, in particular, that $\nu|_{B(0,1)^c}$ has a density with respect to Lebesgue measure. The growth conditions on $\varphi$ are needed to ensure that the solution is conservative, see Lemma~\ref{time-8}. \par \medskip

In the second part of this section we are interested in existence results for Feller processes with decomposable symbols, i.\,e.\ symbols of the form \begin{equation*}
	p(x,\xi) = \sum_{j=1}^n \varphi_j(x) \psi_j(\xi)
\end{equation*}
for continuous $\varphi_j: \mbb{R}^d \to [0,\infty)$ and continuous negative definite mappings $\psi_j: \mbb{R}^d \to \mathbb{C}$. Hoh \cite{hoh94} proved the well-posedness of the associated martingale problem for real-valued $\psi_j$ satisfying the growth condition \begin{equation*}
		1+ \sum_{j=1}^n \psi_j(\xi) \geq c (1+|\xi|^{\alpha}), \qquad |\xi| \geq 1
	\end{equation*}
for some $c,\alpha>0$ under the assumption that $\varphi_j \in C^{2(m+1)+d}_b(\mbb{R}^d)$ for some $m>\tfrac{d}{2\alpha}$.  More recently, Feller processes with decomposable symbols were studied by Kolokoltsov \cite{kol04}; his main result requires that $\varphi_j \in C^{s}(\mbb{R}^d)$ for some $s>2+d/2$ and that $\varphi_j$ is bounded if the support of the L\'evy measure $\nu_j$ of $\psi_j$ is not bounded, i.\,e.\ $\nu_j(B(0,R)^c)>0$ for all $R>0$.  If we consider, for instance, isotropic stable processes, $\psi_j(\xi)= |\xi|^{\alpha_j}$, this means that both Hoh and Kolokoltsov assume boundedness of the coefficients $\psi_j$ as well as a certain regularity of $\varphi_j$. Combining the results from the previous section with a classical perturbation result, we will prove an existence result for unbounded continuous functions $\varphi_j$.  \par
Recall that an operator $(P,\mc{D}(P))$ is called \emph{relatively bounded with respect to an operator $(A,\mc{D}(A))$} if $\mc{D}(P) \supseteq \mc{D}(A)$ and there exist constants $c_1,c_2>0$ such that \begin{equation*}
	\|Pu\| \leq c_1 \|Au\| + c_2 \|u\| \fa u \in \mc{D}(A);
\end{equation*}
$c_1$ is called \emph{relative bound}. A classical perturbation result, see e.\,g.\ \cite[Corollary 3.8]{davies}, states that if $(A,\mc{D}(A))$ is a generator of a Feller semigroup and $(P,\mc{D}(P))$ is relatively bounded with respect to $A$ with relative bound $c_1<1$, then $P+A$ is the generator of a Feller semigroup.

\begin{thm} \label{app-5}
	Let $\psi_i \sim (b_i,Q_i,\nu_i)$, $i \in \{1,2\}$, be two continuous negative definite functions, and suppose that there exists a Bernstein function $f$ such that $\psi_2(\xi) = f(\psi_1(\xi))$, $\xi \in \mbb{R}^d$ and $\lim_{\lambda \to \infty} \lambda^{-1} f(\lambda)=0$. Let $\varphi_i: \mbb{R}^d \to (0,\infty)$, $i \in \{1,2\}$, be continuous functions growing at most quadratically. Suppose that $\varphi_i,\psi_i$ satisfy the growth conditions (i)-(v) in Corollary~\ref{time-15} for $i \in \{1,2\}$. If $\sup_{x \in \mbb{R}^d} \varphi_2(x)/\varphi_1(x) < \infty$ then there exists a conservative rich Feller process with symbol \begin{equation*}
		p(x,\xi) := \varphi_1(x) \psi_1(\xi) + \varphi_2(x) \psi_2(\xi), \qquad x,\xi \in \mbb{R}^d.
	\end{equation*}
\end{thm}

Let us remark that Theorem~\ref{app-5} can be formulated (and proved) in a similar fashion for finite sums $p(x,\xi) = \sum_{i=1}^n \varphi_i(x) \psi_i(\xi)$.

\begin{proof}
	Since $s := \varphi_2/\varphi_1$ is a bounded continuous function which is strictly positive, there exists a rich Feller process with symbol $q(x,\xi) := s(x) \psi_2(\xi)$ and generator $(A,\mc{D}(A))$, cf.\ Theorem~\ref{time-3}.  If we denote by $A_i$ the generator of the L\'evy process with characteristic exponent $\psi_i$, $i \in \{1,2\}$, then $Au = s \cdot A_2 u$ for $u \in \mc{D}(A_2)$. Moreover, $\psi_2 = f \circ \psi_1$ implies $\mc{D}(A_1) \subseteq \mc{D}(A_2)$, and as $\lim_{\lambda \to \infty} \lambda^{-1} f(\lambda)=0$ it follows from \cite[pp.~208, (13.17)]{bernstein} that for any $\delta>0$ there exists a constant $c>0$ such that \begin{equation*}
		\|A_2 u\|_{\infty} \leq \delta \|A_1 u\|_{\infty} + c \|u\|_{\infty}, \qquad u \in \mc{D}(A_1) \subseteq \mc{D}(A_2).
	\end{equation*}
	Hence, \begin{equation*}
		\|Au\|_{\infty}
		= \sup_{x \in \mbb{R}^d} |s(x) A_2 u(x)| 
		\leq \|s\|_{\infty} \delta \|A_1 u\|_{\infty} + \|s\|_{\infty} c \|u\|_{\infty}, \qquad u \in \mc{D}(A_1).
	\end{equation*}
	Choosing $\delta>0$ sufficiently small, we find that $A$ is relatively bounded with respect to $A_1$ with relative bound strictly smaller than $1$. By \cite[Corollary 3.8]{davies} , there exists a rich Feller process with generator $A+A_1$ and symbol \begin{equation*}
		q(x,\xi) + \psi_1(\xi) = \frac{\varphi_2(x)}{\varphi_1(x)} \psi_2(\xi) + \psi_1(\xi).
	\end{equation*}
	Since $C_c^{\infty}(\mbb{R}^d)$ is a core for $(A_1,\mc{D}(A_1))$, it follows easily from the above estimate for $\|Au\|$ that $C_c^{\infty}(\mbb{R}^d)$ is a core for $A+A_1$. Applying Theorem~\ref{time-9}, we find that there exists a conservative rich Feller process with symbol \begin{equation*}
		\varphi_1(x) (q(x,\xi) + \psi_1(\xi)) = \varphi_2(x) \psi_2(\xi) + \varphi_1(x) \psi_1(\xi) = p(x,\xi). \qedhere
	\end{equation*}
\end{proof}

For two given continuous negative definite functions $\psi_1, \psi_2$ it is, in general, hard to check whether there exists a Bernstein function $f$ such that $\psi_2 = f \circ \psi_1$. To our best knowledge, there is no general result which gives sufficient and/or necessary conditions for such a decomposition of $\psi_2$. We close this section with some examples.

\begin{bsp} \label{app-7}
	Let $\varphi_1,\varphi_2: \mbb{R}^d \to (0,\infty)$ be continuous functions such that $\sup_x \tfrac{\varphi_2(x)}{\varphi_1(x)}<\infty$. \begin{enumerate}
		\item isotropic $\alpha$-stable: Suppose that there exist $0<\alpha < \beta \leq 2$ such that $\varphi_1(x) \leq c_1 (1+|x|^{\beta})$ and $\varphi_2(x) \leq c_2(1+|x|^{\alpha})$ for absolute constants $c_1,c_2>0$. Then there exists a rich conservative Feller process with symbol \begin{equation*} 
			p(x,\xi) = \varphi_1(x) |\xi|^{\beta} + \varphi_2(x) |\xi|^{\alpha}, \qquad x,\xi \in \mbb{R}^d.
		\end{equation*}
		\item relativistic stable: If $\varphi_i(x) \leq c(1+|x|^2)$, $i \in \{1,2\}$, then there exists a conservative rich Feller process with symbol \begin{equation*}
			p(x,\xi) = \varphi_1(x) \big((|\xi|^2+m^2)^{\beta/2}-m^{\beta} \big) + \varphi_2(x) \big((|\xi|^2+m^2)^{\alpha/2}-m^{\alpha} \big), \quad x,\xi \in \mbb{R}^d
		\end{equation*}
		for any $\alpha,\beta \in (0,2)$, $\alpha \leq \beta$.
		\item homographic: If $\varphi_i(x) \leq c(1+|x|^2)$, $i \in \{1,2\}$, then there exists a conservative rich Feller process with symbol \begin{equation*}
			p(x,\xi) = \varphi_1(x) \frac{\varrho |\xi|^2}{1+\varrho |\xi|^2} + \varphi_2(x) \frac{\lambda |\xi|^2}{1+\lambda |\xi|^2}, \qquad x,\xi \in \mbb{R}^d,
		\end{equation*}
		for any $\lambda,\varrho \in (0,\infty)$, $\varrho \leq \lambda$.
	\end{enumerate}
\end{bsp}

\appendix
\section{Appendix} \label{appen}

Let $(q(x,\cdot))_{x \in \mbb{R}^d}$ be a family of continuous negative definite functions $q(x,\cdot): \mbb{R}^d \to \mbb{C}$ with characteristics $(b,Q,\nu)$, i.\,e.\
 \begin{equation*}
	q(x,\xi) = q(x,0) -ib(x) \cdot \xi + \frac{1}{2} \xi \cdot Q(x) \xi + \int_{y \neq 0} (1-e^{iy \cdot \xi}+i y \cdot \xi \I_{(0,1)}(|y|)) \, \nu(x,dy), \quad x,\xi \in \mbb{R}^d.
\end{equation*}
In this section we establish an equivalent characterization of continuity of the mapping $x \mapsto q(x,\xi)$ in terms of the characteristics $(b,Q,\nu)$. To this end, it is useful to replace $\I_{(0,1)}$ by a smooth cut-off function $\chi$ satisfying $\I_{B(0,1)} \leq \chi \leq \I_{B(0,2)}$: \begin{equation*}
	q(x,\xi) = q(x,0) -i\tilde{b}(x) \cdot \xi + \frac{1}{2} \xi \cdot Q(x) \xi + \int_{y \neq 0} (1-e^{iy \cdot \xi}+i y \cdot \xi\chi(y)) \, \nu(x,dy), \quad x,\xi \in \mbb{R}^d;
\end{equation*}
here $\tilde{b}(x) := b(x) - \int_{y \neq 0} (\I_{(0,1)}(|y|)-\chi(y)) \, \nu(x,dy)$ is the compensated drift.

\begin{thm} \label{appen-1}
	Let $(q(x,\cdot))_{x \in \mbb{R}^d}$ be a family of continuous negative definite functions of the form \begin{equation}
		q(x,\xi) = q(x,0) -i b(x) \cdot \xi + \int_{y \neq 0} (1-e^{iy \cdot \xi}+iy \cdot \xi \chi(y)) \, \nu(x,dy) \label{appen-eq3}
	\end{equation}
	for a cut-off function $\chi \in C_c^{\infty}(\mbb{R}^d)$, $\I_{B(0,1)} \leq \chi \leq \I_{B(0,2)}$. If $q$ is locally bounded, then the following statements are equivalent. \begin{enumerate}
		\item $x \mapsto q(x,\xi)-q(x,0)$ is continuous for all $\xi \in \mbb{R}^d$.
		\item Each of the following conditions is satisfied. \begin{enumerate}
			\item $x \mapsto q(x,\xi)-q(x,0)$ is continuous for all $\xi \in \mbb{R}^d$.
			\item $\lim_{|\xi| \to 0} \sup_{x \in K} |q(x,\xi)-q(x,0)|=0$ for any compact set $K \subseteq \mbb{R}^d$.
		\end{enumerate}
		\item Each of the following conditions is satisfied. \begin{enumerate}
		\item $x \mapsto b(x)$ is continuous.
		\item $x \mapsto \nu(x,\cdot)$ is vaguely continuous on $(\mbb{R}^d \backslash \{0\},\mc{B}(\mbb{R}^d \backslash \{0\}))$, i.\,e.\ \begin{equation*}
			\int_{y \neq 0} f(y) \, \nu(z,dy) \xrightarrow[]{z \to x} \int_{y \neq 0} f(y) \, \nu(x,dy) \fa f \in C_c(\mbb{R}^d \backslash \{0\}), x \in \mbb{R}^d.
		\end{equation*}
		\item $\lim_{R \to \infty} \sup_{x \in K} \nu(x, B(0,R)^c) = 0$ for any compact set $K \subseteq \mbb{R}^d$ and $R>0$.
		\item $\lim_{r \to 0} \sup_{x \in K} \int_{|y| \leq r} |y|^2 \, \nu(x,dy)=0$ for any compact set $K \subseteq \mbb{R}^d$ and $r>0$.
	\end{enumerate}
	\item The pseudo-differential operator with symbol $-(q(x,\xi)-q(x,0))$ maps $C_c^{\infty}(\mbb{R}^d)$ into $C(\mbb{R}^d)$.  \end{enumerate}
\end{thm}

\begin{bem} \label{appen-2} \begin{itemize}
	\item (i) $\iff$ (ii) $\iff$ (iv) $\implies$ (iii) remain valid for any Borel measurable cut-off function $\chi$ such that $\I_{B(0,1)} \leq \chi \leq \I_{B(0,2)}$.
	\item Continuity of $x \mapsto q(x,\xi)-q(x,0)$ does \emph{not} imply that the diffusion coefficient $Q(\cdot)$ is continuous; consider for instance \begin{equation*}
		q(x,\xi) := \frac{1}{2} \xi^2 \I_{\{0\}}(x) + \frac{1-\cos(x \xi)}{x^2} \I_{\mbb{R} \backslash \{0\}}(x), \qquad x,\xi \in \mbb{R}.
	\end{equation*}
	However, the proof of Theorem~\ref{appen-1} shows that the implications (i) $\iff$ (ii) $\iff$ (iv) $\implies$ (iii)(a)-(c) remain valid if $q$ has a non-vanishing diffusion part $Q$.
	\item Denote by $A$ the pseudo-differential operator with symbol $-q$. Continuity of the symbol with respect to $x$ is a sufficient but not necessary condition for the continuity of $x \mapsto Af(x)$, $f \in C_c^{\infty}(\mbb{R}^d)$, see \cite[Example 4.6]{rs98} for a (counter)example. Schilling \cite[Theorem 4.4]{rs98} has shown that if $A(C_c^{\infty}(\mbb{R}^d)) \subseteq C(\mbb{R}^d)$ then \begin{equation}
		\forall \xi \in \mbb{R}^d: x \mapsto q(x,\xi) \, \, \text{continuous} \iff x \mapsto q(x,0) \, \, \text{is continuous}. \label{appen-eq4}
	\end{equation}
	\item A symbol $q$ of the form \eqref{appen-eq3} is locally bounded if, and only if, for any compact set $K \subseteq \mbb{R}^d$ there exists a constant $c>0$ such that $|q(x,\xi)| \leq c(1+|\xi|^2)$ for all $x \in K$, $\xi \in \mbb{R}^d$. By \cite[Lemma 2.1, Remark 2.2]{rs-grow}, this is equivalent to \begin{equation}
		\forall K \subseteq \mbb{R}^d \, \, \text{cpt.}: \, \, \, \sup_{x \in K} |q(x,0)| +\sup_{x \in K} |b(x)| + \sup_{x \in K} \int_{y \neq 0} (|y|^2 \wedge 1) \, \nu(x,dy) <\infty.  \label{appen-eq5}
	\end{equation}
	Local boundedness of $q$ is, in particular, satisfied if there exists a rich Feller process with symbol $q$, cf.\ \cite[Proposition 2.27(d)]{ltp} and Corollary~\ref{appen-25}.
	\end{itemize}
\end{bem}

\begin{proof}[Proof of Theorem~\ref{appen-1}]
	To keep notation simple, we prove the result only in dimension $d=1$. Without loss of generality we may assume that $q(x,0)=0$ (otherwise replace $q(x,\xi)$ by $q(x,\xi)-q(x,0)$). The implications (i) $\implies$ (ii) and (iv) $\implies$ (i) follow from \cite[Theorem 4.4]{rs98}. Moreover, if $x \mapsto q(x,\xi)$ is continuous, then we find from the local boundedness of $q$ and the dominated convergence theorem that \begin{equation*}
		  x \mapsto - \int_{\mbb{R}} q(x,\xi) \hat{f}(\xi) e^{ix \xi} \, d\xi
		\end{equation*}
		is continuous for all $f \in C_c^{\infty}(\mbb{R})$, and this proves (iv) $\implies$ (i). In the remaining part of the proof we show that (ii) $\implies$ (iii) $\implies$ (i). \par \medskip
	
	(iii) $\implies$ (i): It suffices to show that \begin{equation*}
		x \mapsto p(x,\xi): = \int_{y \neq 0} (1-e^{iy \xi} + i y \xi \chi(y)) \, \nu(x,dy)
	\end{equation*}
	is continuous for all $\xi \in \mbb{R}$. Clearly, \begin{equation*}
		|p(x,\xi)-p(z,\xi)| \leq I_1 + I_2+I_3
	\end{equation*}
	where \begin{align*}
		I_1 &:= \frac{|\xi|^2}{2} \left( \int_{|y| \leq r} |y|^2 \, \nu(x,dy) + \int_{|y| \leq r} |y|^2 \, \nu(z,dy) \right) \\
		I_2 &:= \left| \int_{r<|y|<R} (1-e^{iy \xi}+iy \xi \chi(y)) \, \nu(x,dy) - \int_{r<|y|<R} (1-e^{iy \xi}+iy \xi \chi(y)) \, \nu(z,dy) \right| \\
		I_3 &:= \nu(x,B(0,R)^c) + \nu(z,B(0,R)^c).
	\end{align*}
	The vague continuity implies that $I_2 \to 0$ as $z \to x$ for fixed $r,R>0$. Letting first $z \to x$ and then $r \to 0$ and $R \to \infty$, it follows from (iii)(c) and (iii)(d) that $p(\cdot,\xi)$ is continuous. \par \medskip
	
	(ii) $\implies$ (iii): Denote by $A$ the pseudo-differential operator with symbol $-q$. Exactly the same reasoning as in \cite[proof of Theorem 4.4]{rs98} shows that (iii)(c) holds. For $\varphi \in C_c^{\infty}(\mbb{R})$ and $x \in \mbb{R}$ define \begin{equation*}
		S_x(\varphi) := A(|\cdot-x|^2 \varphi(\cdot-x))(x) = \int_{y \neq 0} y^2 \varphi(y) \, \nu(x,dy).
	\end{equation*}
	If we denote by $\mc{F}f := \hat{f}$ the Fourier transform of a function $f$, then \begin{equation*}
		\mc{F}(|\cdot-x|^2 \varphi(\cdot-x))(\xi) = e^{-ix \xi} \mc{F}(|\cdot|^2 \varphi(\cdot))(\xi), \qquad \xi \in \mbb{R}.
	\end{equation*}
	Since $q$ is locally bounded and $x \mapsto q(x,\xi)$ is continuous for all $\xi \in \mbb{R}$, an application of the dominated convergence theorem shows that \begin{equation*}
		x \mapsto S_x(\varphi) = - \int_{\mbb{R}} q(x,\xi) \mc{F}(|\cdot-x|^2 \varphi(\cdot-x))(\xi) e^{ix \xi} \, d\xi
	\end{equation*}
	is continuous. Choose $\varphi_k \in C_c^{\infty}(\mbb{R})$ such that $\I_{B(0,1/k)} \leq \varphi_k \leq \I_{B(0,2/k)}$ and $\varphi_{k+1} \leq \varphi_k$. Then \begin{equation*}
		S_x(\varphi_k) \leq \int_{|y| \leq 2/k} |y|^2 \, \nu(x,dy) \xrightarrow[]{k \to \infty} 0 \fa x \in \mbb{R}
	\end{equation*}
	and $S_x(\varphi_k) \geq S_x(\varphi_{k+1})$. Applying Dini's theorem, we find that \begin{equation*}
		\sup_{x \in K} \int_{|y| \leq 1/k} |y|^2 \, \nu(x,dy) \leq \sup_{x \in K} \int y^2 \varphi_k(y) \, \nu(x,dy) = \sup_{x \in K} |S_x(\varphi_k)|\xrightarrow[\text{Dini}]{k \to \infty} 0
	\end{equation*}
	for any compact set $K$, and this proves (iii)(d).  If we set $\mu(x,dy) := |y-x|^2 \, \nu(x,dy+x)$, then \begin{equation}
		T_x(\varphi) := A(|\cdot-x|^2 \varphi(\cdot))(x) =\int_{y \neq 0} |y|^2 \varphi(x+y) \, \nu(x,dy) = \int \varphi(y) \, \mu(x,dy) \label{appen-eq6}
	\end{equation}
	for all $\varphi \in C_c^{\infty}(\mbb{R})$. Since \begin{align*}
		\mc{F}(|\cdot-x|^2 \varphi(\cdot))(\xi) &= \frac{1}{2\pi} \int_{\mbb{R}} (y-x)^2 \varphi(y) e^{-iy \xi} \, dy \\ &=  \frac{1}{2\pi} \left( \int_{\mbb{R}} y^2 \varphi(y) \,e^{-iy \xi} \, dy -2 x \int_{\mbb{R}} y \varphi(y) e^{-iy \xi} \, dy + x^2 \int_{\mbb{R}} \varphi(y) e^{-iy \xi} \, dy \right)
	\end{align*}
	there exists for any compact set $K$ an integrable function $g$ such that $\sup_{x \in K} |\mc{F}(|\cdot-x|^2 \varphi)(\xi)| \leq g(\xi)$ for all $\xi \in \mbb{R}$. As $x \mapsto q(x,\xi)$ is continuous and locally bounded, the dominated convergence theorem shows that the mapping \begin{equation*}
		x \mapsto T_x(\varphi) = - \int_{\mbb{R}} q(x,\xi) \mc{F}(|\cdot-x|^2 \varphi(\cdot))(\xi) e^{ix \xi} \, d\xi
	\end{equation*} 
	is continuous for all $\varphi \in C_c^{\infty}(\mbb{R})$. By \eqref{appen-eq6} this means that \begin{equation*}
		\int_{\mbb{R}} \varphi(y) \, \mu(z,dy) \xrightarrow[]{z \to x} \int_{\mbb{R}} \varphi(y) \, \mu(x,dy) \fa f \in C_c^{\infty}(\mbb{R}), x \in \mbb{R}.
	\end{equation*}
	Combining this with the fact that the local boundedness of $q$ implies \begin{equation}
		\sup_{z \in K} \int_{|y|\leq R} \, \mu(z,dy)<\infty \fa R>0, K \subseteq \mbb{R} \,\, \text{cpt.} \label{appen-eq7}
	\end{equation}
	cf.\ \eqref{appen-eq5}, we conclude that $\mu(\cdot,dy)$ is vaguely continuous on $(\mbb{R},\mc{B}(\mbb{R}))$.  Using that \begin{align*}
		\left| \int_{\mbb{R}} \varphi(y-x) \, \mu(x,dy) - \int_{\mbb{R}} \varphi(y-z) \, \mu(z,dy) \right|
		&\leq \int_{\mbb{R}} |\varphi(y-x)-\varphi(y-z)| \, \mu(z,dy) \\
		&\quad + \left| \int_{\mbb{R}} \varphi(y-x) \, \mu(z,dy)- \int_{\mbb{R}} \varphi(y-x) \, \mu(x,dy) \right|
	\end{align*}
	for all $\varphi \in C_c(\mbb{R})$, it follows from \eqref{appen-eq7} and the vague continuity of $\mu(x,\cdot)$ that \begin{align*}
		\int_{\mbb{R}} \varphi(y-z) \, \mu(z,dy) \xrightarrow[]{z \to x} \int_{\mbb{R}} \varphi(y-x) \, \mu(x,dy)
	\end{align*}
	for all $\varphi \in C_c(\mbb{R})$. Since $\nu(x,dy) = \frac{1}{|y|^2} \mu(x,dy-x)$, it is not difficult to see that this implies that $\nu(z,\cdot)$ converges vaguely on $(\mbb{R} \backslash \{0\},\mc{B}(\mbb{R} \backslash \{0\}))$ to $\nu(x,\cdot)$ as $z \to x$. To prove (iii)(a) we note that \begin{align*}
		b(x) = A((\cdot-x) \chi(\cdot-x) )(x) 
		&= - \int_{\mbb{R}} q(x,\xi) e^{ix \xi} \mc{F}(((\cdot-x) \chi(\cdot-x))(\xi) \, d\xi \\
		&= - \int_{\mbb{R}} q(x,\xi) \mc{F}(\id(\cdot) \chi(\cdot))(\xi) \, d\xi;
	\end{align*}
	here $\id(y) := y$. Applying the dominated convergence theorem another time, we find that $x \mapsto b(x)$ is continuous. 
\end{proof}

As a direct consequence of Theorem~\ref{appen-1} we obtain the following corollary.

\begin{kor} \label{appen-25}
	Let $(X_t)_{t \geq 0}$ be a rich Feller process with symbol $q$ and characteristics $(b,Q,\nu)$.  If $x \mapsto q(x,0)$ is continuous, then $q \mapsto q(x,\xi)$ is continuous for all $\xi \in \mbb{R}^d$ and (iii)(a)-(c) in Theorem~\ref{appen-1} are satisfied. If additionally $Q=0$, then (iii)(d) holds true.
\end{kor}

\begin{proof}
	By  \cite[Proposition 2.27(d)]{ltp} the symbol $q$ is locally bounded. Denote by $A$ the pseudo-differential operator with symbol $-q$. Since $A$ is, when restricted to $C_c^{\infty}(\mbb{R}^d)$, the generator of $(X_t)_{t \geq 0}$, we have $A(C_c^{\infty}(\mbb{R}^d)) \subseteq C(\mbb{R}^d)$, and therefore \eqref{appen-eq4} shows that $x \mapsto q(x,\xi)$ is continuous for all $\xi \in \mbb{R}^d$. Now the assertion follows from Theorem~\ref{appen-1} and Remark~\ref{appen-2}.
\end{proof}

Corollary~\ref{appen-3} is crucial for the proof of Lemma~\ref{time-9} since it shows that the truncated symbol $q_R$ is continuous with respect to the space variable $x$.

\begin{kor} \label{appen-3}
	Let $(q(x,\cdot))_{x \in \mbb{R}^d}$ be a family of continuous negative definite functions, \begin{equation*}
		q(x,\xi) =- i b(x) \xi + \frac{1}{2} \xi \cdot Q(x) \xi + \int_{y \neq 0} (1-e^{iy \cdot \xi}+iy \cdot \xi \I_{(0,1)}(|y|)) \, \nu(x,dy),
	\end{equation*}
	such that $x \mapsto q(x,\xi)$ is continuous for all $\xi \in \mbb{R}^d$ and $q$ is locally bounded. Let $R: \mbb{R} \to [1,\infty)$ and $\chi: \mbb{R} \to [0,1]$ be continuous mappings such that $\I_{B(0,1)} \leq \chi \leq \I_{B(0,2)}$. Then \begin{equation*}
			q_R(x,\xi) := -i b(x) \cdot \xi + \frac{1}{2} \xi \cdot Q(x) \xi  + \int_{y \neq 0} \chi \left( \frac{y}{R(x)} \right) (1-e^{iy \cdot \xi}+iy \cdot \xi \I_{(0,1)}(|y|)) \, \nu(x,dy) 
		\end{equation*}
		is continuous with respect to $x$.
\end{kor}

Note that the assumption of local boundedness is, in particular, satisfied if there exists a rich Feller process with symbol $q$, cf.\ \cite[Proposition 2.27(d)]{ltp}.

\begin{proof}[Proof of Corollary~\ref{appen-3}]
	Fix $\xi \in \mbb{R}^d$. Since $q(\cdot,\xi)$ is continuous, it suffices to show that \begin{equation*}
		f(x) := q_R(x,\xi)-q(x,\xi) = \int_{|y| \geq 1} \left[ \chi \left( \frac{y}{R(x)} \right)-1 \right] (1-e^{iy \cdot \xi}) \, \nu(x,dy)
	\end{equation*}
	is continuous. It follows from Theorem~\ref{appen-1} and Remark~\ref{appen-2} that $\nu$ satisfies \ref{appen-1}(iii)(b),(c). By the triangle inequality, $|f(x)-f(z)| \leq I_1+I_2$ where \begin{align*}
		I_1 &:= \left| \int_{|y| \geq 1} \left[ \chi \left( \frac{y}{R(x)} \right)- \chi \left( \frac{y}{R(z)} \right) \right] (1-e^{iy \cdot \xi}) \, \nu(z,dy) \right| \\
		I_2 &:= \left|\int_{|y| \geq 1} \left[ \chi \left( \frac{y}{R(x)} \right)-1 \right] (1-e^{iy \cdot \xi}) \, \nu(x,dy) - \int_{|y| \geq 1} \left[ \chi \left( \frac{y}{R(x)} \right)-1 \right] (1-e^{iy \cdot \xi}) \, \nu(z,dy) \right|.
	\end{align*}
	Choose cut-off functions $\varphi_k \in C_c(\mbb{R}^d)$ such that $\I_{B(0,k)} \leq \varphi_k \leq \I_{B(0,2k)}$. Then $I_2$ is bounded above by  \begin{align*}
		&\left|\int_{|y| \geq 1} \left[ \chi \left( \frac{y}{R(x)} \right)-1 \right] \varphi_k(y) (1-e^{iy \cdot \xi}) \, \nu(x,dy) - \int_{|y| \geq 1} \left[ \chi \left( \frac{y}{R(x)} \right)-1 \right] \varphi_k(y) (1-e^{iy \cdot \xi}) \, \nu(z,dy) \right| \\
		&\quad + 4 \int_{|y| \geq k} \nu(x,dy) + 4 \int_{|y| \geq k} \, \nu(z,dy).
	\end{align*}
	Letting first $z \to x$ and then $k \to \infty$ it follows from \ref{appen-1}(iii)(b) and \ref{appen-1}(iii)(c) that $I_2 \to 0$. Similarly, \begin{align*}
		I_1
		&\leq \int_{1 \leq |y| \leq k} \left| \chi \left( \frac{y}{R(x)} \right)- \chi \left( \frac{y}{R(z)} \right) \right| \, \nu(z,dy)  + 4 \int_{|y| \geq k} \nu(z,dy).
	\end{align*}
	Using the local boundedness of $q$, cf.\ \eqref{appen-eq5}, and the fact that $\chi$ is uniformly continuous, it follows easily that the first term on the right-hand side converges to $0$ as $z \to x$. The second term converges to $0$ as $k \to \infty$ uniformly on compact sets. Letting first $z \to x$ and then $k \to \infty$ we get $I_1 \to 0$.
\end{proof}

\begin{ack}
	I am grateful to Ren\'e Schilling for helpful comments and discussions.
\end{ack}

\end{document}